\documentclass{amsart}
\usepackage[utf8]{inputenc}
\usepackage{url}
\usepackage{graphicx}
\usepackage{subcaption}
\usepackage{geometry}
\usepackage{enumerate}
\usepackage{tikz}
\usetikzlibrary{positioning}

\usepackage[round,comma]{natbib} %
\usepackage[colorlinks,urlcolor=red,citecolor=blue,linkcolor=red]{hyperref}

\usepackage{amsmath}
\usepackage{amsthm}
\usepackage{amsfonts}
\usepackage{amssymb}
\usepackage{dsfont}
\usepackage{nicefrac}
\usepackage{booktabs}
\usepackage{mathrsfs}
\usepackage{bm} 
\usepackage{mathtools}      
\usepackage{mathabx}

\newcommand{\cB}{{\mathcal B}}

\newcommand{\cF}{{\mathcal F}}
\newcommand{\cG}{{\mathcal G}}

\newcommand{\cN}{{\mathcal N}}
\newcommand{\cS}{{\mathcal S}}

\newcommand{\cU}{{\mathcal U}}

\DeclareMathOperator{\var}{Var}

\newcommand{\cov}{\operatorname{Cov}}

\newcommand{\Ind}{{\mathds 1}}

\newcommand{\essinf}{\mathop{\mathrm{ess\,inf}}\limits}

\DeclarePairedDelimiter{\abs}{\lvert}{\rvert}

\DeclareMathOperator{\supp}{supp}

\newcommand{\R}{\mathbb{R}}
\newcommand{\C}{\mathbb{C}}

\newcommand{\bbF}{\mathbb{F}}

\newcommand{\bbG}{\mathbb{G}}

\newtheorem{theorem}{Theorem}[section]
\newtheorem{corollary}[theorem]{Corollary}      %
\newtheorem{lemma}[theorem]{Lemma}              %
\newtheorem{proposition}[theorem]{Proposition}  %
\theoremstyle{definition}
\newtheorem{example}[theorem]{Example} %
\newtheorem{definition}[theorem]{Definition} %
\newtheorem{remark}[theorem]{Remark}%
\newtheorem{assumption}[theorem]{Assumption}%
\numberwithin{equation}{section}
\usepackage[backgroundcolor=white,bordercolor=orange]{todonotes}

\title{An extended CIR process with stochastic discontinuities}

\author{Claudio Fontana}
\address{University of Padova, Department of Mathematics}
\email{fontana@math.unipd.it}

\author{Simone Pavarana}
\address{University of Freiburg, Department of stochastic mathematics}
\email{simone.pavarana@stochastik.uni-freiburg.de}

\author{Thorsten Schmidt}
\address{University of Freiburg, Department of stochastic mathematics}
\email{thorsten.schmidt@stochastik.uni-freiburg.de}

\begin{document}

\begin{abstract}
We study an extension of the Cox–Ingersoll–Ross (CIR) process that incorporates jumps at deterministic dates, referred to as \emph{stochastic discontinuities}. Our main motivation stems from short-rate modelling in the context of overnight rates, which often exhibit spikes at predetermined dates corresponding to central bank meetings. We provide a formal definition of a CIR process with stochastic discontinuities, where the jump sizes depend on the pre-jump state, thereby allowing for both upward and downward movements as well as potential autocorrelation among jumps. Under mild assumptions, we establish existence of such a process and identify sufficient and necessary conditions under which the process inherits the affine property of its continuous counterpart. We illustrate our results with practical examples that generate both upward and downward jumps while preserving affinity and non-negativity. In particular, we show that a stochastically discontinuous CIR process can be constructed by applying a deterministic c\`adl\`ag time-change to a standard CIR process. Finally, we further enrich the affine framework by characterizing conditions that ensure infinite divisibility of the extended CIR process.
\end{abstract}


\maketitle
\begin{center}
\textit{(Preliminary version)}
\end{center}

\section{Introduction}

Jumps at predetermined dates, often driven by monetary policy decisions or liquidity constraints, have emerged as a distinctive feature of overnight rates such as SOFR (Secured Overnight Financing Rate) in the US; see Figure~\ref{fig:SOFR}. These scheduled jumps introduce \emph{stochastic discontinuities} (i.e., discontinuities occurring at ex-ante known points in time) into the dynamics of overnight rates. These features cannot be adequately captured by classical short-rate models, motivating new approaches. The work of \cite{Fontana2024} provides a rigorous treatment of interest rate modelling with stochastic discontinuities, addressing the issue of absence of arbitrage within the HJM framework and discussing pricing and hedging in the setting of affine semimartingales, as introduced by \cite{keller-ressel2019}. In particular, they extend the classical Hull--White model by incorporating stochastic discontinuities with independent Gaussian jump sizes. Related approaches with a more empirical focus have been proposed by \cite{BackwellHayes21} and \cite{BraceGellertSchloegl22}. While we focus here on jumps at deterministic dates to remain within the affine framework, in principle, predictable jump times could also be treated following the approach of \cite{GehmlichSchmidt2018} and \cite{FontanaSchmidt2018}.

Following the inflationary shock and the onset of the COVID-19 pandemic in 2020, central banks implemented rapid and substantial rate hikes, moving away from the negative interest rate regime that had characterized the previous decade. This development motivates the use of short-rate models that preclude negative interest rates, for which the Cox--Ingersoll--Ross (CIR) model is a natural candidate. 

The contribution of this work is to extend the classical univariate CIR framework by incorporating stochastic discontinuities, where the distribution of jump sizes depends explicitly on the pre-jump state. In contrast to the standard affine jump–diffusion (AJD) framework commonly employed in default intensity modeling (see, e.g., \cite{duffie2001risk}), our construction permits both upward and downward jumps in the state process. The scope of potential financial applications is not limited to interest rate modeling; it naturally extends to settings where scheduled jumps occur and non-negativity of the process is essential, such as in volatility modeling within the Heston framework. In this preliminary version of the paper, we focus on establishing the mathematical properties of the model — most notably its affine structure — while leaving a detailed study of applications to future research.

The paper is organized as follows. In Section~\ref{sec:1}, we introduce the notion of an extended CIR process, with particular emphasis on its state space. Under mild assumptions, we establish existence and uniqueness of a strong solution to the corresponding SDE. Section~\ref{sec:2} is devoted to the affine property: we derive necessary and sufficient conditions ensuring that the extended CIR process remains both affine and non-negative. We further illustrate the theory with two constructive examples and provide numerical simulations that highlight the occurrence of both upward and downward jumps. In particular, we demonstrate that the affine property in a semimartingale setting is preserved under deterministic time changes, and that stochastic discontinuities can be incorporated into a CIR dynamics through a deterministic c\`adl\`ag time change. Section~\ref{sec:3} concludes the analysis by further extending the affine framework, establishing conditions under which the extended CIR process is an infinitely divisible semimartingale.

\begin{figure}[t]
    \centering
\includegraphics[width=0.9\textwidth]{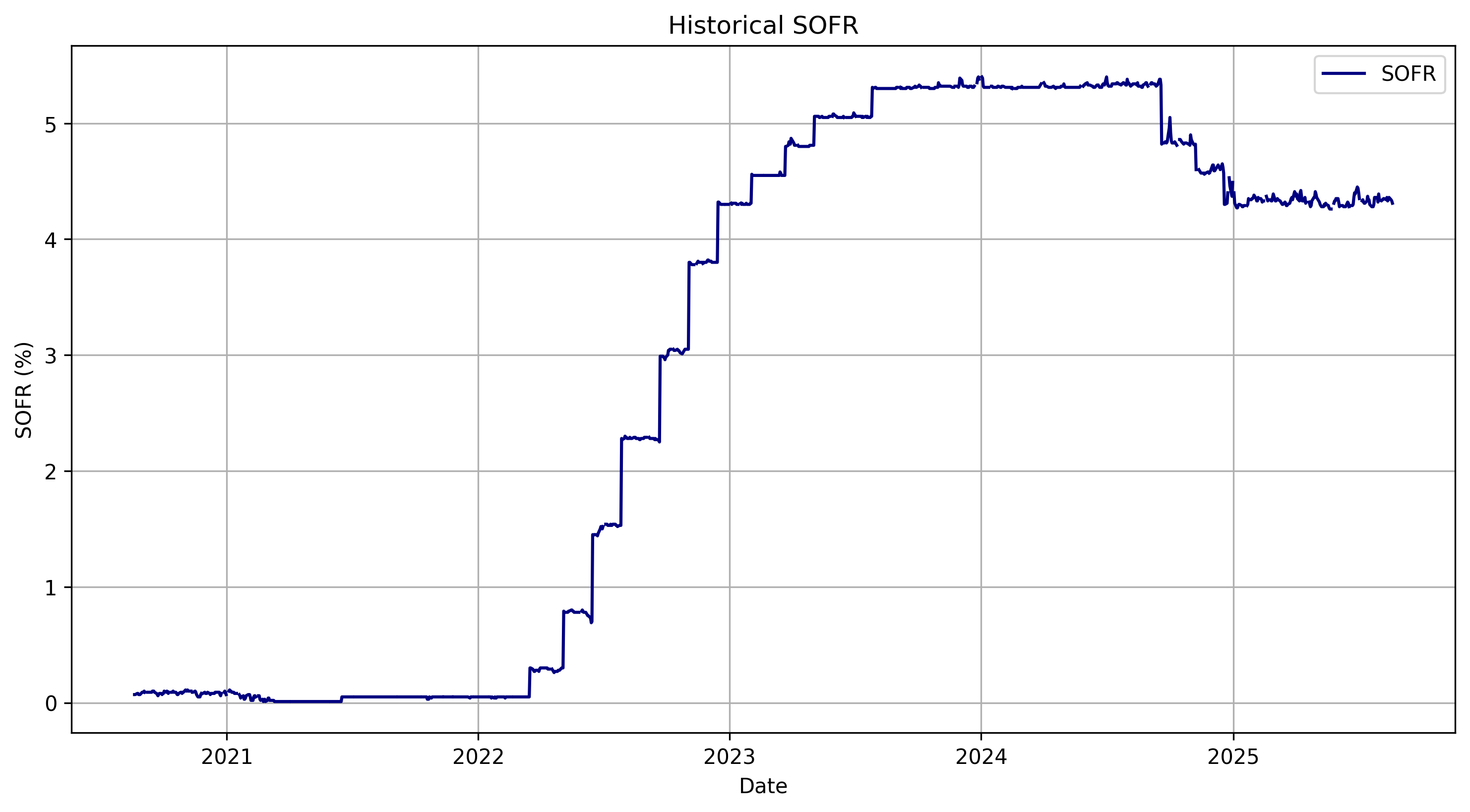}
 \caption{SOFR time series exhibiting scheduled jumps at predetermined dates. 
    Source: Federal Reserve Bank of New York, \emph{Secured Overnight Financing Rate (SOFR)}, 
    retrieved from FRED, Federal Reserve Bank of St. Louis; \url{https://fred.stlouisfed.org/series/SOFR}, August 21, 2025.}
    \label{fig:SOFR}
\end{figure}

\section{The extended CIR process}\label{sec:1}

In this section, we introduce the extended CIR process as the non-negative solution of an SDE that combines the standard CIR dynamics with a pure jump component. The jumps occur at a sequence of deterministic dates, with sizes depending on the pre-jump state. The jump component can be expressed in terms of its associated random measure, which allows us to rewrite the SDE as driven by two martingale parts: a continuous one given by the Brownian motion and a discontinuous one driven by the compensated jump measure. A key distinction from classical jump dynamics based on Poisson processes is that the compensator is not absolutely continuous, although it remains deterministic. We establish existence and uniqueness of a strong solution to this SDE and conclude the section by showing that the state dependence of the jump sizes naturally induces serial correlation among the jumps.

We work on a complete stochastic basis 
\( (\Omega, \mathcal{F}, \mathbb{F}, P) \) (see Definitions~1.2 and~1.3 in \cite{JacodShiryaev}), where the filtration \( \mathbb{F} = (\mathcal{F}_t)_{t \geq 0} \) supports both a standard Brownian motion \( W \) and a sequence of i.i.d.\ random variables \( (Z_n)_{n \geq 1} \), uniformly distributed on \([0,1]\) and independent of \( W \). Fix parameters \(\theta \geq 0\) (long-term mean), \(\kappa \geq 0\) (speed of mean reversion), and \(\sigma \geq 0\) (volatility), and let \((F_n)_{n\geq 1}\) be a family of measurable functions 
\[
F_n : \mathbb{R}_+ \times [0,1] \to \mathbb{R}.
\]
Moreover, let \( \mathcal{S} = (s_n)_{n \geq 1} \) be an increasing sequence of deterministic jump times.

We now introduce the extended CIR process with stochastic discontinuities.

\begin{definition}[Extended CIR process]\label{def:CIR}
A stochastic process \( X \) is called an \emph{extended CIR process} if
\[
P(X_t \geq 0 \;\; \text{for all } t \geq 0) = 1,
\]
and \( X \) is a strong solution to
\begin{equation}\label{eq:CIR}
    dX_t = \kappa (\theta - X_t) \, dt + \sigma \sqrt{X_t} \, dW_t + dJ_t, \qquad t \geq 0,
\end{equation}
with initial condition \(X_0 = x \geq 0\), where the jump component is given by
\[
J_t = \sum_{n:\,s_n \leq t} F_n(X_{s_n-}, Z_n), \qquad t \geq 0.
\]
\end{definition}

\begin{remark}
The functions \( F_n \) determine how the jump sizes depend on the pre-jump value of \( X \). Defining
\begin{equation}\label{eq:jump_size}
\xi_n = F_n(X_{s_n-}, Z_n), \quad n \geq 1,
\end{equation}
the conditional distribution of the jumps can be derived from \( F_n \) by standard change-of-variable arguments.

From a modeling perspective, as illustrated in the examples below (see Section \ref{sec:ex}), it is often more convenient to specify directly the conditional distribution of the jump size \( \xi_n \). In this case, given
\[
p_n(x,A) = P(\xi_n \in A \mid X_{s_n-} = x), 
\qquad A \in \mathcal{B}(\mathbb{R}), \; x \in \mathbb{R}_+,
\]
the corresponding function \( F_n \) satisfying \eqref{eq:jump_size} can be written as
\[
F_n(x,z) = G_n^{\leftarrow}(x,z),
\]
where \( G_n(x,\cdot) \) denotes the (regular) conditional distribution function
\[
G_n(x,y) := p_n(x,(-\infty,y]), \quad y \in \mathbb{R},
\]
and \( G_n^{\leftarrow}(x,\cdot) \) its generalized inverse, defined by
\[
G_n^{\leftarrow}(x,z) := \inf \{ y \in \mathbb{R} : G_n(x,y) \ge z \}, \quad z \in (0,1).
\]
See \cite[Proposition~7.2]{McNeil2015Quantitative}.
\end{remark}

\subsection{Existence of the extended CIR process}

The well-posedness of Definition~\ref{def:CIR} requires establishing the existence of a non-negative strong solution to the SDE~\eqref{eq:CIR}. To this end, we impose the following standing assumption. 

\begin{assumption}[No accumulation of jump times]\label{ass:accumulation}
The sequence of jump times \( \mathcal{S}=(s_n)_{n\geq 1} \) satisfies the following limiting behavior:
\[
s_n \to \infty \quad \text{as } n \to \infty .
\]
\end{assumption}

\begin{remark}
In contrast to most models in the literature, which incorporate only finitely many scheduled jumps in the short-rate dynamics (see, e.g., \cite{Piazzesi2001}, \cite{BraceGellertSchloegl22}, and \cite{Fontana2024}), our approach allows for countably many jump times, provided that they do not accumulate in finite time. 
\end{remark}

A second requirement for our results is that the jumps do not lead the process outside the state space. This is formalized in the following definition.

\begin{definition}[Admissibility of jump sizes]\label{def:support}
A sequence of jump size functions \( (F_n)_{n \geq 1} \) is called \emph{admissible} if it satisfies
\[
\essinf_{z \in [0,1]} F_n(x,z) \geq -x, 
\qquad \text{for all } x \in \R_+, \; n \geq 1.
\]
\end{definition}

\begin{remark}
Note that in the case of almost surely upward jumps, as in standard AJD models where \(F_n(x,z) \ge 0\) for all \(x \in \R_+\) and almost every \(z \in [0,1]\), the state space is never violated, and the dependence on the pre-jump state may be ignored.  

To allow for downward jumps, the dependence on the pre-jump state becomes crucial. The admissibility condition ensures that the process remains non-negative almost surely.
\end{remark}

We can now state the strong existence result.

\begin{theorem}\label{th:existence}
Suppose that Assumption~\ref{ass:accumulation} holds and that the sequence $(F_n)_{n\geq 1}$ is admissible. Then, for any initial condition \( X_0 = x \geq 0 \), there exists a unique strong solution \( X^x = (X^x_t)_{t \geq 0} \) to the SDE~\eqref{eq:CIR}, which remains non-negative almost surely:
\[
P\big(X_t \geq 0 \;\; \text{for all } t \geq 0\big) = 1.
\]
\end{theorem}

\begin{proof}
In view of Assumption~\ref{ass:accumulation}, we construct the solution to the SDE~\eqref{eq:CIR} with initial condition \( X_0 = x \geq 0 \) in a piecewise manner. First, consider the dynamics on the interval \( [0, s_1) \), where no jumps occur:
\[
dX_t = \kappa (\theta - X_t) \, dt + \sigma \sqrt{X_t} \, dW_t, \qquad t \in [0, s_1).
\]
By the Yamada--Watanabe theorem for one-dimensional diffusions (see \cite[Proposition~2.13]{KaratzasShreve1988}), the SDE on \( [0, s_1) \) admits a unique non-negative strong solution \( \hat{X}^x = (\hat{X}^x_t)_{t \in [0, s_1)} \).

Proceeding inductively, for each \( n \geq 1 \) we consider the interval \( [s_n, s_{n+1}) \), where the dynamics is given by
\[
dX_t = \kappa (\theta - X_t) \, dt + \sigma \sqrt{X_t} \, dW_t, \qquad t \in (s_n, s_{n+1}),
\]
with initial condition
\[
X_{s_n} = X^{X_{s_{n-1}}}_{s_n-} + F_n(X_{s_n-},Z_n).
\]
By Definition~\ref{def:support}, 
\[
F_n(x,z) \geq -x, \qquad \text{for all } x \in \mathbb{R}_+ \text{ and a.e.\ } z \in [0,1],
\]
so that
\begin{align*}
P(X_{s_n} \geq 0) 
&= P\!\left(F_n(X^{X_{s_{n-1}}}_{s_n-},Z_n)\geq -X^{X_{s_{n-1}}}_{s_n-}\right) \\
&= E\!\left[P\!\left(F_n(X^{X_{s_{n-1}}}_{s_n-},Z_n)\geq -X^{X_{s_{n-1}}}_{s_n-}\,\middle|\, \mathcal{F}_{s_n-}\right)\right] \\
&= 1,
\end{align*}
for all \( n \geq 1 \). Hence,
\[
\operatorname{supp}(X_{s_n}) \subseteq \mathbb{R}_+.
\]
By the Yamada--Watanabe argument, the SDE on \( [s_n, s_{n+1}) \) again admits a unique non-negative strong solution, denoted \( \hat{X}^{X_{s_n}} = (\hat{X}^{X_{s_n}}_t)_{t \in [s_n, s_{n+1})} \).

The full process \( X^x \) solving \eqref{eq:CIR} is then obtained by patching together the local solutions:
\[
X^x_t = \hat{X}^x_t \Ind_{[0, s_1)}(t) + \sum_{n \geq 1} \hat{X}^{X_{s_n}}_t \Ind_{[s_n, s_{n+1})}(t), \quad \text{for all } t \geq 0.
\]
This completes the proof.
\end{proof}

\begin{remark}
Theorem~\ref{th:existence} guarantees the existence of a non-negative solution to the SDE~\eqref{eq:CIR}. To ensure strict positivity, one must additionally impose the Feller condition
\[
2\theta\kappa \geq \sigma^2,
\]
which is sufficient to ensure that the continuous CIR process remains strictly positive (see Proposition~1.2.15 in \cite{Alfonsi2015Affine}). If, in addition, the strict inequality in Definition~\ref{def:support} holds — thereby preventing the process from jumping to zero — then the proof of Theorem~\ref{th:existence} can be easily adapted to yield a strictly positive solution to the SDE~\eqref{eq:CIR}.
\end{remark}

Modeling the jump sizes via the transportation maps \( F_n \) and the sequence \( (Z_n)_{n \geq 1} \) of i.i.d.\ uniform random variables, independent of \( W \), ensures that the solution process is adapted to the natural filtration generated by the driving noise processes. Since this filtration is contained in \( \mathbb{F} \), the process is a strong solution to the SDE~\eqref{eq:CIR}. 

\begin{remark}
Without the transportation maps \( F_n \), it would not be possible to ensure adaptedness of the solution. In that case, one would typically resort to a weak solution framework and interpret the SDE as a martingale problem; see, for instance, \cite[Section~III.2]{JacodShiryaev} and \cite{criens23martingale} for the case with stochastic discontinuities. 
\end{remark}

Notably, one can rewrite the SDE as:
\[
dX_t = \kappa (\theta - X_t) \, dt + \sigma \sqrt{X_t} \, dW_t + \int_0^1 F_t(X_{t-}, z)\, \mu(dt, dz), \quad t \geq 0,
\]
where \( \mu \) is the random measure associated with the scheduled jumps, defined as
\begin{equation}\label{eq:rm}
\mu(dt, dz) = \sum_{n \geq 1} \delta_{(s_n, Z_n)}(dt, dz).
\end{equation}
By Theorem~1.8 in \cite{JacodShiryaev}, there exists a predictable random measure 
\[
\nu(dt, dz) = K(t,dz)\, dA_t, 
\] 
where \( A = (A_t)_{t \geq 0} \) is a predictable, increasing, finite-variation process, and \( K:\R_+\times \mathcal{B}([0,1]) \to \mathbb{R}_+ \) is a stochastic kernel, such that the compensated jump term
\[
\int_0^{\cdot} \int_0^1 F_t(X_{t-}, z)\, \big(\mu(dt, dz) - \nu(dt, dz)\big)
\]
is a local martingale. 

Thus, the solution to the SDE~\eqref{eq:CIR} can be interpreted as the solution to an SDE driven by a continuous local martingale (the Brownian term) and a purely discontinuous local martingale (the compensated jump term); see Definition~2.24(a) in \cite{JacodShiryaev}.

In the present setting, it is straightforward to check that, in analogy with the case of Poisson random measures, the predictable compensator \(\nu\) reduces to a deterministic measure.

\begin{proposition}\label{prop:compensator}
The predictable compensator $\nu$ of the random measure $\mu$ defined in \eqref{eq:rm} is
\[
\nu(dt,dz) = \Ind_{[0,1]}(z)\,dz \, dN(t),
\]
where $N:\mathbb{R}_+\to\mathbb{N}_0$ denotes the deterministic counting process of jump times, that is,
\[
N(t)=\sum_{n\ge 1}\Ind_{\{s_n\le t\}}, \qquad t\ge 0.
\]
\end{proposition}

\begin{proof}
By Theorem~1.8 in \cite{JacodShiryaev}, a random measure $\nu$ is the predictable compensator of $\mu$ if and only if, for every non-negative predictable function $H:\Omega\times\R_+\times [0,1]\to\R_+$, one has
\begin{equation}\label{eq:compensator}
E\Bigg[\int_0^\infty\int_0^1 H_t(z)\,\mu(dt,dz)\Bigg]
=E\Bigg[\int_0^\infty\int_0^1 H_t(z)\,\nu(dt,dz)\Bigg].
\end{equation}
Consider $\nu$ as given in the statement. In particular,
\[
\nu(dt,dz)=\sum_{n\geq 1}\Ind_{[0,1]}(z)\,dz \,\delta_{s_n}(dt).
\]
Substituting this into the right-hand side of \eqref{eq:compensator} yields  
\[
E\Bigg[\int_0^\infty\int_0^1 H_t(z)\,\nu(dt,dz)\Bigg]
=E\Bigg[\sum_{n\geq 1}\int_0^\infty\int_0^1 H_t(z)\Ind_{[0,1]}(z)\,dz\,\delta_{s_n}(dt)\Bigg].
\]
By the sifting property of Dirac measures, we obtain  
\[
E\Bigg[\int_0^\infty\int_0^1 H_t(z)\,\nu(dt,dz)\Bigg]
=E\Bigg[\sum_{n\geq 1}\int_0^1 H_{s_n}(z)\Ind_{[0,1]}(z)\,dz\Bigg]. 
\]
The right-hand side can be rewritten as  
\[
E\Bigg[\sum_{n\geq 1}\int_0^1 H_{s_n}(z)\Ind_{[0,1]}(z)\,dz\Bigg]
=E\Bigg[\sum_{n\geq 1}E\big[H_{s_n}(Z_n)\Ind_{[0,1]}(Z_n)\mid\cF_{s_n-}\big]\Bigg],
\]
which leads to  
\begin{align*}
E\Bigg[\int_0^\infty\int_0^1 H_t(z)\,\nu(dt,dz)\Bigg]
&=E\Bigg[\sum_{n\geq 1}E\big[H_{s_n}(Z_n)\Ind_{[0,1]}(Z_n)\mid\cF_{s_n-}\big]\Bigg]\\
&=E\Bigg[\sum_{n\geq 1}H_{s_n}(Z_n)\Ind_{[0,1]}(Z_n)\Bigg]\\
&=E\Bigg[\int_0^\infty\int_0^1 H_t(z)\,\mu(dt,dz)\Bigg],
\end{align*}
where the last equality follows again from the sifting property of Dirac measures together with the definition of $\mu$.

Hence, $\nu$ satisfies \eqref{eq:compensator}, and thus it is the compensator of $\mu$, unique up to sets of $P$-null measure.
\end{proof}

To conclude this section, we highlight a key difference with the extended Hull--White model of \cite{Fontana2024}, which generalises the classical Hull--White framework by adding a jump component with deterministic jump times and independent Gaussian jump sizes. In contrast, the extended CIR dynamics given by \eqref{eq:CIR} naturally allow for a non-trivial correlation structure between jumps.

To formalise this, let $c(n,m)$ denote the covariance between the jumps at times $s_n$ and $s_m$, i.e.
\[
c(n,m)\coloneqq\cov\big(F_n(X_{s_n-},Z_n),F_m(X_{s_m-},Z_m)\big).
\]

\begin{proposition}\label{prop:cov_jump}
For all $n,m\geq 1$, the covariance function $c(n,m)$ is given by
\[
c(n,m)=\cov\big(\mu_n(X_{s_n-}),\mu_m(X_{s_m-})\big),
\]
where 
\[
\mu_n(x)\coloneqq \int_0^1 F_n(x,z)\,dz, 
\qquad x\in\R_+, \; n\geq 1.
\]
\end{proposition}

\begin{proof}
For notational convenience, set
\[
\xi_n \coloneqq F_n(X_{s_n-},Z_n), \qquad n \geq 1.
\]
Then the covariance can be written as
\begin{equation}\label{eq:cov}
c(n,m) \;=\; E[\xi_n \xi_m] - E[\xi_n]\,E[\xi_m], 
\qquad n,m \geq 1.
\end{equation}

We first analyze the second term. For every \(n \geq 1\), one has
\begin{align*}
E[\xi_n] = E\big[ F_n(X_{s_n-},Z_n) \big] = E\Big[ E\big[ F_n(X_{s_n-},Z_n) \,\big|\, \mathcal{F}_{s_n-} \big] \Big] = E\Bigg[ \int_0^1 F_n(X_{s_n-},z)\,dz \Bigg] = E\big[ \mu_n(X_{s_n-}) \big].
\end{align*}
Hence the product of expectations appearing in \eqref{eq:cov} takes the form
\begin{equation}\label{eq:cov1}
E[\xi_n]\,E[\xi_m] 
= E\big[\mu_n(X_{s_n-})\big]\,E\big[\mu_m(X_{s_m-})\big].
\end{equation}

For the first term in \eqref{eq:cov}, assume without loss of generality that \(n \geq m\). Then
\[
E[\xi_n \xi_m] 
= E\big[ F_n(X_{s_n-},Z_n)\,F_m(X_{s_m-},Z_m) \big] 
= E\Big[ E\big[ F_n(X_{s_n-},Z_n)\,F_m(X_{s_m-},Z_m) \,\big|\, \mathcal{F}_{s_n-} \big] \Big].
\]
Since \(s_n \geq s_m\), the random variable \(F_m(X_{s_m-},Z_m)\) is \(\mathcal{F}_{s_n-}\)-measurable. Therefore,
\[
E[\xi_n \xi_m] 
= E\Big[ F_m(X_{s_m-},Z_m)\,E\big[ F_n(X_{s_n-},Z_n) \,\big|\, \mathcal{F}_{s_n-} \big] \Big]
= E\big[ F_m(X_{s_m-},Z_m)\,\mu_n(X_{s_n-}) \big].
\]

Taking the conditional expectation with respect to \(\mathcal{F}_{s_m-}\), we obtain
\begin{align*}
E[\xi_n \xi_m] 
&= E\Big[ E\big[ F_m(X_{s_m-},Z_m)\,\mu_n(X_{s_n-}) \,\big|\, \mathcal{F}_{s_m-} \big] \Big] \\
&= E\Big[ E\big[ F_m(X_{s_m-},Z_m) \,\big|\, \mathcal{F}_{s_m-} \big] \, E\big[ \mu_n(X_{s_n-}) \,\big|\, \mathcal{F}_{s_m-} \big] \Big] \\
&= E\big[ \mu_m(X_{s_m-}) \, E[ \mu_n(X_{s_n-}) \mid \mathcal{F}_{s_m-}] \big],
\end{align*}
where the second equality follows from the mutual independence of the sequence \((Z_n)_{n \geq 1}\).

Finally, we conclude that
\[
E[\xi_n \xi_m] 
= E\big[ \mu_m(X_{s_m-}) \, E[ \mu_n(X_{s_n-}) \mid \mathcal{F}_{s_m-}] \big] 
= E\big[ \mu_n(X_{s_n-}) \, \mu_m(X_{s_m-}) \big].
\]
A symmetric argument shows that the same expression holds when \(m \geq n\).

Substituting this identity together with \eqref{eq:cov1} into \eqref{eq:cov} yields the desired result.
\end{proof}

The expression for \( c(n,m) \) in Proposition~\ref{prop:cov_jump} is in general non-zero, showing that jump sizes may exhibit serial correlation. In Section~\ref{sec:2}, we present an explicit example in which the covariance of jumps can be computed, illustrating how a dependence structure naturally arises through their dependence on the pre-jump state.

\section{The affine framework for extended CIR processes}\label{sec:2}

In this section, we investigate the conditions under which an extended CIR process is affine. Affine processes are Markov processes whose conditional characteristic function can be expressed in an exponential affine form; for a rigorous treatment, see \cite{DuffieFilipovicSchachermayer}. In financial applications, affine processes are widely used, particularly in interest rate modeling (see, e.g., \cite{filipovic2009term} and \cite{Fontana22}). 

The references above focus on stochastically continuous processes. In the presence of stochastic discontinuities, \cite{keller-ressel2019} introduced the concept of \emph{affine semimartingales}, which naturally extends the theory of affine processes to account for jumps at predictable times.

Let \( D \subseteq \mathbb{R}^d \) be a closed convex cone of full dimension, meaning it is a convex set closed under multiplication by positive scalars and has a linear hull equal to \( \mathbb{R}^d \). The complex dual cone space \( \mathcal{U} \) associated with \( D \) via the Fourier transform is given by:
\[
\mathcal{U} \coloneqq \{ u \in \mathbb{C}^d : \langle \Re(u), x \rangle \leq 0 \text{ for all } x \in D \}.
\]
A canonical example of a state space \( D \) is \( \mathbb{R}^m_+ \times \mathbb{R}^n \) with \( m+n=d \), whose dual cone is \( \mathbb{C}^m_- \times i\mathbb{R}^n \), where \( \mathbb{C}_- \coloneqq \{ u \in \mathbb{C} : \Re(u) \leq 0 \} \).

A semimartingale $X$ taking values in \( D \) is called  \emph{affine} if there exist \( \mathbb{C} \)- and \( \mathbb{C}^d \)-valued functions \( \phi_t(T,u) \) and \( \psi_t(T,u) \), continuous in \( u \in \cU \) and satisfying \( \phi_t(T,0) = \psi_t(T,0) = 0 \), such that:
\begin{equation}\label{eq:affine}
E \left[ e^{\langle u, X_T \rangle} \mid \mathscr{F}_t \right] = \exp\left(\phi_t(T,u) + \langle \psi_t(T,u), X_t \rangle\right), \quad 0 \leq t \leq T, \quad u \in \mathcal{U}.
\end{equation}

An affine semimartingale is called \emph{quasiregular} (see \cite[Definition 2.5]{keller-ressel2019}) if the following conditions hold:
\begin{itemize}
    \item[(i)] The functions \( \phi \) and \( \psi \) are of finite variation in \( t \) and are c\`adl\`ag in both \( t \) and \( T \).
    \item[(ii)] For all \( 0 < t \leq T \), the functions
    \[
    u \mapsto \phi_{t-}(T,u), \quad\text{and}\quad u \mapsto \psi_{t-}(T,u)
    \]
    are continuous on \( \mathcal{U} \).
\end{itemize}

Since the continuous CIR process is affine, a natural question arises: under which conditions on the jump size functions $(F_n)_{n\geq 1}$ does the extended CIR process remain affine in the semimartingale sense? To answer this, we proceed in two steps. First, we provide sufficient conditions that guarantee the affine property, illustrated with practical examples showing how to construct both upward and downward jumps while preserving affinity and respecting the state space. The first example consists of a process that jumps to zero and then adds a non-negative affine random variable. The second example is obtained by time-changing a continuous CIR process with a deterministic c\`adl\`ag  time change. In particular, we show that the affine property is preserved under such deterministic time changes. Finally, we derive necessary conditions, yielding a complete characterization of affine extended CIR processes.

\subsection{Sufficient conditions for the affine property}

It is reasonable to expect that, if the jump sizes are specified in an exponential–affine form, the affine structure of the classical CIR diffusion extends to its discontinuous counterpart.

Before proving this formally, we record a criterion that links an affine specification of the jump distribution to the admissible–jump condition in Definition~\ref{def:support}.

\begin{lemma}\label{lem:affine_supp}
Suppose that Assumption~\ref{ass:accumulation} holds and that the jump sizes \( (F_n)_{n \geq 1} \) have a conditional characteristic function of the form  
\begin{equation}\label{eq:affine_jump}
E \left[ e^{u F_n(X_{s_n-},Z_n)} \,\middle|\, \mathcal{F}_{s_n-} \right]
  = \exp\bigl( \gamma_{n,0}(u) + \gamma_{n,1}(u)\,X_{s_n-} \bigr),
  \qquad u\in i\mathbb{R},
\end{equation}
for some continuous functions \( \gamma_{n,0},\gamma_{n,1} : i\mathbb{R}\to\mathbb{C} \).  
Then the sequence $(F_n)_{n\geq 1}$ is admissible if and only if the following conditions hold for every \( n\ge 1 \):
\begin{enumerate}
  \item[(i)]  Both \( \gamma_{n,0} \) and \( \gamma_{n,1} \) admit analytic extensions to the left half‑plane \( \mathbb{C}_{-} \). Moreover, there exists a locally bounded map \( C : \mathbb{R}_{+}\to\mathbb{R}_{>0} \) such that
    \[
    \Re\left(\gamma_{n,0}(w)\right) \leq |w| \limsup_{x \to 0^+} C(x), \quad 
    \Re\left(\gamma_{n,1}(w)\right) \leq |w| \limsup_{x \to \infty} \frac{C(x)}{x},
    \]
    for all \( w \in \mathbb{C}_{-} \).
  \item[(ii)] The following asymptotic conditions are satisfied:
    \[
    \lim_{y\to\infty}\frac{\gamma_{n,0}(-y)}{y}\le 0,
    \quad
    \lim_{y\to\infty}\frac{\gamma_{n,1}(-y)}{y}\le 1 .
    \]
\end{enumerate}
\end{lemma}
\begin{proof}
Suppose that the sequence $(F_n)_{n\geq 1}$ is admissible. Then, by Theorem~\ref{th:existence}, there exists a unique non-negative strong solution \( X = (X_t)_{t \geq 0} \) to the SDE~\eqref{eq:CIR}.

Moreover, the admissibility condition in Definition~\ref{def:support} implies that
\[
\operatorname{Leb}\bigl(\{z \in [0,1] : F_n(x,z) < -x\}\bigr) = 0, 
\quad \text{for all } x \in \mathbb{R}_+,
\]
so that the distribution of \( F_n(x, Z_n) \) is supported on \( [-x, \infty) \) for every \( x \in \mathbb{R}_+ \).

It then follows from Corollary~\ref{cor:Lukacs} that the conditional characteristic function \eqref{eq:affine_jump} admits an analytic extension to the left half-plane \( \mathbb{C}_{-} \), and satisfies the exponential bound
\begin{equation}\label{eq:upper_bound}
\left| \exp\left( \gamma_{n,0}(w) + \gamma_{n,1}(w)\,x \right) \right| \leq e^{C(x) \lvert w \rvert}, \quad \text{for all } w \in \mathbb{C}_{-},\, x\in\R_+
\end{equation}
for some locally bounded function \( C: \mathbb{R}_+ \to \mathbb{R}_{>0} \) depending on the pre-jump state.

This implies that both \( \gamma_{n,0} \) and \( \gamma_{n,1} \) admit analytic extensions to \( \mathbb{C}_{-} \), and their real parts satisfy
\begin{equation}\label{eq:real_bound}
\Re\left(\gamma_{n,0}(w)\right) + \Re\left(\gamma_{n,1}(w)\right) x \leq C(x) \lvert w \rvert, \quad \text{for all } w \in \mathbb{C}_{-},\ x \in \mathbb{R}_+.
\end{equation}

Since the inequality holds for all \( x \in \mathbb{R}_+ \), we may fix \( \epsilon > 0 \) and restrict it to \( x \in [0, \epsilon) \), yielding
\[
\Re\left(\gamma_{n,0}(w)\right) \leq C(x) \lvert w \rvert - \Re\left(\gamma_{n,1}(w)\right) x, \quad \text{for all } w \in \mathbb{C}_{-} \text{ and } x \in [0, \epsilon).
\]

Taking the supremum over \( x \in [0, \epsilon) \), we obtain
\[
\Re\left(\gamma_{n,0}(w)\right) \leq \sup_{x\in[0,\epsilon)} \left\{ C(x) \lvert w \rvert - \Re\left(\gamma_{n,1}(w)\right) x \right\}, \quad \text{for all } w \in \mathbb{C}_{-},
\]
which yields the desired upper bound for \( \Re\left(\gamma_{n,0}(w)\right) \) in Condition~(i) of the lemma by letting \( \epsilon \to 0^+ \).

Similarly, from \eqref{eq:real_bound}, for any \( M > 0 \), we have
\[
\Re\left(\gamma_{n,1}(w)\right) \leq \frac{C(x)}{x} \lvert w \rvert - \frac{\Re\left(\gamma_{n,0}(w)\right)}{x}, \quad \text{for all } w \in \mathbb{C}_{-}, \, x \geq M. 
\]
Taking the supremum over \( x \geq M \) and letting \( M \to \infty \) then gives the upper bound for \( \Re\left(\gamma_{n,1}(w)\right) \) in Condition~(i).

To establish Condition~(ii), we apply the explicit formula for the infimum of a distribution bounded from below given in Corollary~\ref{cor:Lukacs}, which yields
\[
\inf\supp F_n(x,Z_n) = - \lim_{y \to \infty} \frac{\gamma_{n,0}(-y) + \gamma_{n,1}(-y)\,x}{y}, \quad \text{for all } x \in \mathbb{R}_+.
\]
Using the identity
\[
\inf\supp F_n(x,Z_n) = \essinf_{z \in [0,1]} F_n(x,z),
\]
and substituting the expression above into the condition from Assumption~\ref{def:support}, we obtain
\[
\lim_{y \to \infty} \frac{\gamma_{n,0}(-y)}{y} + x \left( \lim_{y \to \infty} \frac{\gamma_{n,1}(-y)}{y} - 1 \right) \leq 0, \quad \text{for all } x \in \mathbb{R}_+,
\]
which implies Condition~(ii) of the lemma by the same argument used to establish the upper bounds in Condition~(i).

The converse implication follows directly from the converse direction in Corollary~\ref{cor:Lukacs}, by reversing the steps of the preceding argument.
\end{proof}

\begin{remark}\label{rem:bochner}
In the statement of Lemma~\ref{lem:affine_supp}, it is implicitly assumed that Equation~\eqref{eq:affine_jump} defines a well-defined characteristic function. By Bochner’s theorem, this means that the function
\[
\varphi_n(u) \coloneqq \exp \left( \gamma_{n,0}(u) + \gamma_{n,1}(u) X_{s_n-} \right), \quad u \in i\mathbb{R},\; n \geq 1,
\]
must be positive definite, continuous at the origin, and normalized such that
\begin{equation}\label{eq:normalization}
\varphi_n(0) = 1.
\end{equation}
Continuity at the origin is automatically satisfied, as the functions \( \gamma_{n,0} \) and \( \gamma_{n,1} \) are assumed continuous on the entire domain. From the normalization condition~\eqref{eq:normalization}, it follows directly that
\[
\gamma_{n,0}(0) = \gamma_{n,1}(0) = 0.
\]

By Schoenberg’s theorem, a necessary and sufficient condition for \( \varphi_n \) to be positive definite is that the functions \( \gamma_{n,0} \) and \( \gamma_{n,1} \) are \emph{conditionally positive definite} (CPD). This requires that
\[
\gamma_{n,0}(u) = \overline{\gamma_{n,0}(-u)}, \quad \gamma_{n,1}(u) = \overline{\gamma_{n,1}(-u)} \quad \text{for all } u \in i\mathbb{R},
\]
and that for all \( u_1, \dots, u_m \in i\mathbb{R} \) and \( c_1, \dots, c_m \in \mathbb{C} \) with \( \sum_{i=1}^{m} c_i = 0 \), the following inequalities hold:
\[
\sum_{i,j=1}^{m} c_i \overline{c_j} \gamma_{n,0}(u_i - u_j) \geq 0, \quad \sum_{i,j=1}^{m} c_i \overline{c_j} \gamma_{n,1}(u_i - u_j) \geq 0.
\]
\end{remark}

We now state sufficient conditions guaranteeing the existence of a quasiregular affine semimartingale solution to the SDE~\eqref{eq:CIR}. 

\begin{theorem}\label{th:suff}
Suppose that Assumption~\ref{ass:accumulation} holds, and that the jump sizes \( (F_n)_{n\geq 1} \) admit a conditional characteristic function of the affine form~\eqref{eq:affine_jump}, where the functions \( \gamma_{n,0} \) and \( \gamma_{n,1} \) satisfy Conditions~(i)–(ii) in Lemma~\ref{lem:affine_supp}. Then there exists a unique quasi-regular, almost surely non-negative affine semimartingale \( X \) solving the SDE~\eqref{eq:CIR}.
\end{theorem}

\begin{proof}
Assume that the jump sizes \( (F_n)_{n \geq 1} \) admit conditional characteristic functions of the affine form \eqref{eq:affine_jump}, where the functions \( \gamma_{n,0} \) and \( \gamma_{n,1} \) are continuous and satisfy Conditions~(i)--(ii) of Lemma~\ref{lem:affine_supp}. Then, by Theorem~\ref{th:existence} and Lemma~\ref{lem:affine_supp}, there exists a unique strong solution \( X \) to the SDE~\eqref{eq:CIR} with almost surely non-negative paths.

To verify that \( X \) satisfies the affine property \eqref{eq:affine}, fix an arbitrary time horizon \( T \geq 0 \), and let \( N \) be the counting function introduced in Proposition~\ref{prop:compensator}. By Assumption~\ref{ass:accumulation}, the number of jump times in any finite interval is finite, so in particular \( N(T) < \infty \). Denote by \( s_{N(T)} \) the last jump time prior to \( T \). On the interval \( [s_{N(T)}, T] \), the process \( X \) evolves without discontinuities. Hence, by the affine property of the continuous CIR process, we have
\[
E \!\left[ e^{u X_T} \mid \mathcal{F}_t \right] 
  = \exp\!\big( \phi^c_t(T, u) + \psi^c_t(T, u) X_t \big), 
  \qquad t \in [s_{N(T)}, T],
\]
where \( \phi^c_t(T, u) \) and \( \psi^c_t(T, u) \) are the Riccati coefficients of the classical CIR model (see \cite[Section~10.3.2.2]{filipovic2009term}).

Applying the law of total expectation and using the affine form \eqref{eq:affine_jump} at time \( s_{N(T)} \), we compute
\begin{align*}
E \left[ e^{uX_T} \mid \mathcal{F}_{s_{N(T)}-} \right] 
&= E \left[ E \left[ e^{uX_T} \mid \mathcal{F}_{s_{N(T)}} \right] \mid \mathcal{F}_{s_{N(T)}-} \right] \\
&= e^{\phi^c_{s_{N(T)}}(T,u)} \, E \left[ e^{\psi^c_{s_{N(T)}}(T,u) X_{s_{N(T)}}} \mid \mathcal{F}_{s_{N(T)}-} \right] \\
&= \exp\left( \phi_{s_{N(T)}-}(T,u) + \psi_{s_{N(T)}-}(T,u) X_{s_{N(T)}-} \right),
\end{align*}
where we define
\begin{align*}
\phi_{s_{N(T)}-}(T,u) &\coloneqq \phi^c_{s_{N(T)}}(T,u) + \gamma_{N(T),0}(\psi^c_{s_{N(T)}}(T,u)), \\
\psi_{s_{N(T)}-}(T,u) &\coloneqq \psi^c_{s_{N(T)}}(T,u) + \gamma_{N(T),1}(\psi^c_{s_{N(T)}}(T,u)). 
\end{align*}
These expressions are well-defined since \( u \mapsto \psi^c_{s_{N(T)}}(T,u) \) takes values in \( \mathbb{C}_{-} \) for \( u \in \mathbb{C}_{-} \) (see Lemma~2.7 in \cite{keller-ressel2019}), and Condition~(i) of Lemma~\ref{lem:affine_supp} ensures that \( \gamma_{N(T),0} \) and \( \gamma_{N(T),1} \) admit analytic extensions to \( \mathbb{C}_{-} \).

Now consider \( t \in [s_{N(T)-1}, s_{N(T)}) \). By the law of iterated expectations,  
\begin{align*}
E \!\left[ e^{uX_T} \mid \mathcal{F}_t \right] 
&= E \!\left[ E \!\left[ e^{uX_T} \mid \mathcal{F}_{s_{N(T)}-} \right] \,\middle|\, \mathcal{F}_t \right] \\
&= e^{\phi_{s_{N(T)}-}(T,u)} \, E \!\left[ e^{\psi_{s_{N(T)}-}(T,u) X_{s_{N(T)}-}} \mid \mathcal{F}_t \right] \\
&= e^{\phi_{s_{N(T)}-}(T,u)} \, e^{\phi^c_t(s_{N(T)}-, \psi_{s_{N(T)}-}(T,u)) 
          + \psi^c_t(s_{N(T)}-, \psi_{s_{N(T)}-}(T,u)) X_t},
\end{align*}
where the last step follows from the fact that no jumps occur on \( [s_{N(T)-1}, s_{N(T)}) \), so that the affine property of the classical CIR process applies.

Hence, for all \( t \in [s_{N(T)-1}, s_{N(T)}) \),
\[
E \!\left[ e^{uX_T} \mid \mathcal{F}_t \right] 
= \exp\!\left( \phi_t(T,u) + \psi_t(T,u) X_t \right),
\]
with
\begin{align*}
\phi_t(T,u) &\coloneqq \phi_{s_{N(T)}-}(T,u) 
              + \phi^c_t(s_{N(T)}-, \psi_{s_{N(T)}-}(T,u)), \\
\psi_t(T,u) &\coloneqq \psi^c_t(s_{N(T)}-, \psi_{s_{N(T)}-}(T,u)).
\end{align*}

Iterating this argument backward through the sequence of jump times 
\(s_n \leq s_{N(T)}\), we arrive at the affine representation
\[
E \!\left[ e^{u X_T} \mid \mathcal{F}_t \right] 
= \exp\!\left(\phi_t(T,u) + \psi_t(T,u)\,X_t\right), 
\qquad t \in [0,T], \; u \in \mathbb{C}_{-},
\]
where the functions \(\phi_t(T,u)\) and \(\psi_t(T,u)\) are defined recursively by
\begin{align*}
\phi_t(T,u) &= 
  \Ind_{\{N(t)=N(T)\}} \, \phi^c_t(T,u) \\
  &\quad + \Ind_{\{N(t)<N(T)\}} 
     \Big( \phi_{s_{N(t)+1}-}(T,u) 
           + \phi^c_t\!\left(s_{N(t)+1}-, \psi_{s_{N(t)+1}-}(T,u)\right) \Big), \\[0.5em]
\psi_t(T,u) &= 
  \Ind_{\{N(t)=N(T)\}} \, \psi^c_t(T,u) \\
  &\quad + \Ind_{\{N(t)<N(T)\}} 
     \psi^c_t\!\left(s_{N(t)+1}-, \psi_{s_{N(t)+1}-}(T,u)\right),
\end{align*}
with the jump updates specified as
\begin{align*}
\phi_{s_{N(t)+1}-}(T,u) &= \phi_{s_{N(t)+1}}(T,u) 
  + \gamma_{N(t)+1,0}\!\left(\psi_{s_{N(t)+1}}(T,u)\right), \\
\psi_{s_{N(t)+1}-}(T,u) &= \psi_{s_{N(t)+1}}(T,u) 
  + \gamma_{N(t)+1,1}\!\left(\psi_{s_{N(t)+1}}(T,u)\right).
\end{align*}

To verify that \( \phi_t(T,0) = \psi_t(T,0) = 0 \) for all \( t \in [0,T] \), we first note that the functions \( \phi^c \) and \( \psi^c \) satisfy the normalization condition  
\begin{equation}\label{eq:zero_cond}
\phi^c_t(s,0) = \psi^c_t(s,0) = 0, \quad \text{for all } s \geq 0,\ t \in [0,s],
\end{equation}
as they represent the Riccati coefficients of the continuous CIR process.

This directly implies that
\begin{align*}
\phi_t(T,0) &= \phi^c_t(T,0) = 0, \\
\psi_t(T,0) &= \psi^c_t(T,0) = 0,
\end{align*}
for all \( t \in [s_{N(T)}, T] \).

Now consider \( t \in [s_{N(T)-1}, s_{N(T)}) \). By Remark~\ref{rem:bochner}, we have 
\[
\gamma_{N(T),0}(0) = \gamma_{N(T),1}(0) = 0, 
\]
which, together with \eqref{eq:zero_cond}, yields
\begin{align*}
\phi_{s_{N(T)}-}(T,0) & = \phi^c_{s_{N(T)}}(T,0) + \gamma_{N(T),0}(\psi^c_{s_{N(T)}}(T,0)) = 0, \\
\psi_{s_{N(T)}-}(T,0) &= \psi^c_{s_{N(T)}}(T,0) + \gamma_{N(T),1}(\psi^c_{s_{N(T)}}(T,0)) = 0.
\end{align*}
Furthermore, using again \eqref{eq:zero_cond}, we obtain \( \phi^c_t(s_{N(T)}-, 0) = \psi^c_t(s_{N(T)}-, 0) = 0 \), so that
\begin{align*}
\phi_t(T,0) &= \phi_{s_{N(T)}-}(T,0) + \phi^c_t(s_{N(T)}-, \psi_{s_{N(T)}-}(T,0)) = 0, \\
\psi_t(T,0) &= \psi^c_t(s_{N(T)}-, \psi_{s_{N(T)}-}(T,0)) = 0,
\end{align*}
for all \( t \in [s_{N(T)-1}, s_{N(T)}) \).

Iterating this argument backward across all jump times \( s_n\leq s_{N(T)} \), we conclude that \( \phi_t(T,0) = \psi_t(T,0) = 0 \) for all \( t \in [0, T] \), thereby establishing the affine property of \( X \).

Finally, to verify the quasiregularity property, note that Condition~(i) of \cite[Definition~2.5]{keller-ressel2019} is satisfied by construction.

To establish Condition~(ii), observe that for all \(s \geq 0\) and \(t \in [0, s]\), the maps
\[
u \mapsto \phi^c_t(s,u), \quad u \mapsto \psi^c_t(s,u)
\]
are continuous on \(\C_{-}\), as they correspond to the characteristic exponents of the standard CIR process. Hence, \(\phi_{t-}(T,u)\) and \(\psi_{t-}(T,u)\) are continuous in \(u\) for all \(t \in (s_{N(T)}, T]\).

Moreover, by continuity of \(\gamma_{n,0}\) and \(\gamma_{n,1}\) for all \(n \geq 1\), the maps
\[
u \mapsto \phi_{s_{N(T)}-}(T,u), \quad u \mapsto \psi_{s_{N(T)}-}(T,u)
\]
are also continuous. Consequently, for all \(t \in (s_{N(T)-1}, s_{N(T)}]\), the maps
\[
u \mapsto \phi_{t-}(T,u), \quad u \mapsto \psi_{t-}(T,u)
\]
are continuous on \(\C_{-}\).

Iterating this argument backward over all jump times \(s_n \le s_{N(T)}\), we deduce that
\[
u \mapsto \phi_{t-}(T,u), \quad u \mapsto \psi_{t-}(T,u)
\]
are continuous for all \(t \in (0, T]\), thereby verifying Condition~(ii) and completing the proof.
\end{proof}

\begin{remark}
The backward recursion for \( \phi_t(T,u) \) and \( \psi_t(T,u) \) corresponds to a special case of the so-called semi-flow property of affine semimartingales (see \cite[Lemma~2.7]{keller-ressel2019}). In particular, the jump conditions we derive are consistent with those described in \cite[Theorem~3.2]{keller-ressel2019}.
\end{remark}

\subsection{Model instances}\label{sec:ex}

In the following, we propose two constructions of scheduled jumps that preserve the affine property, allow for both upward and downward movements, and ensure non-negativity of the process.

\begin{example}[Drop to zero, then up again]\label{ex:CIR_Drop_up}
Let \( F_n(x, z) = -x + f_n(x, z) \), where \( f_n : \mathbb{R}_+ \times [0,1] \to \mathbb{R}_+ \) is a measurable function such that
\[
E \left[ e^{u f_n(X_{s_n-}, Z_n)} \,\middle|\, \mathcal{F}_{s_n-} \right]
= \exp\bigl( A_n(u) + B_n(u)\,X_{s_n-} \bigr),
\qquad u \in \mathbb{C}_{-},
\]
for some continuous \( \mathbb{C} \)-valued functions \( A_n \) and \( B_n \). Then the jump sizes \( (F_n(X_{s_n-}, Z_n))_{n \geq 1} \) have a characteristic function of the affine form \eqref{eq:affine_jump}, with
\[
\gamma_{n,0}(u) = A_n(u), \qquad \gamma_{n,1}(u) = b_n(u) - u.
\]

Since \( f_n(x, z) \geq 0 \) for all \( (x,z) \in \mathbb{R}_+ \times [0,1] \), the post-jump value \( X_{s_n} = f_n(X_{s_n-}, Z_n) \) remains non-negative. In particular, an application of Corollary~\ref{cor:Lukacs} shows that \( a_n \) and \( b_n \) must satisfy Condition~(i) in Lemma~\ref{lem:affine_supp}, as well as the asymptotic conditions
\[
\lim_{y \to \infty} \frac{A_n(-y)}{y} \le 0,
\qquad
\lim_{y \to \infty} \frac{B_n(-y)}{y} \le 0.
\]

In other words, the jump mechanism is constructed by letting the process jump to zero and then adding a non-negative random variable --- possibly dependent on the pre-jump state --- thereby allowing both upward and downward jumps while preserving non-negativity.

A practical example arises by choosing
\[
f_n(x,z) = G_n^{-1}(z; x, \alpha_n, \beta_n, \lambda), \qquad \alpha_n, \lambda > 0, \, \beta_n \ge 0,
\]
where
\[
G_n(z; x, \alpha_n, \beta_n, \lambda) = P(\xi_n \leq z\mid X_{s_n-}=x), \quad \xi_n \sim \operatorname{Gamma}(\alpha_n + \beta_n X_{s_n-}, \lambda).
\]
This choice makes \( f_n(X_{s_n-}, Z_n) \) a Gamma-distributed random variable with shape parameter \( \alpha_n + \beta_n X_{s_n-} \) and rate \( \lambda \), whose conditional characteristic function is
\[
E \left[ e^{u f_n(X_{s_n-}, Z_n)} \,\middle|\, \mathcal{F}_{s_n-} \right]
= \exp\left( -\alpha_n \log\left(1 - \frac{u}{\lambda}\right) + \beta_n \log\left(1 - \frac{u}{\lambda}\right)\,X_{s_n-} \right),
\qquad u \in \mathbb{C}_{-}.
\]

Note that the complex logarithm is well-defined for all \( u \in \mathbb{C}_{-} \), since writing \( u = -a + ib \), with \( a \geq 0 \), yields
\[
\left| 1 - \frac{u}{\lambda} \right| = \sqrt{\left(1 + \frac{a}{\lambda}\right)^2 + \frac{b^2}{\lambda^2}} \geq 1.
\]

The case \( \beta_n = 0 \) corresponds to a Gamma-distributed post-jump value that is independent of the pre-jump state.

A simulation of a CIR process with stochastic discontinuities constructed in this manner is displayed in Figure~\ref{fig:CIR_Drop_up}.

\end{example}

\begin{figure}[t]
    \centering
    \includegraphics[width=0.9\textwidth]{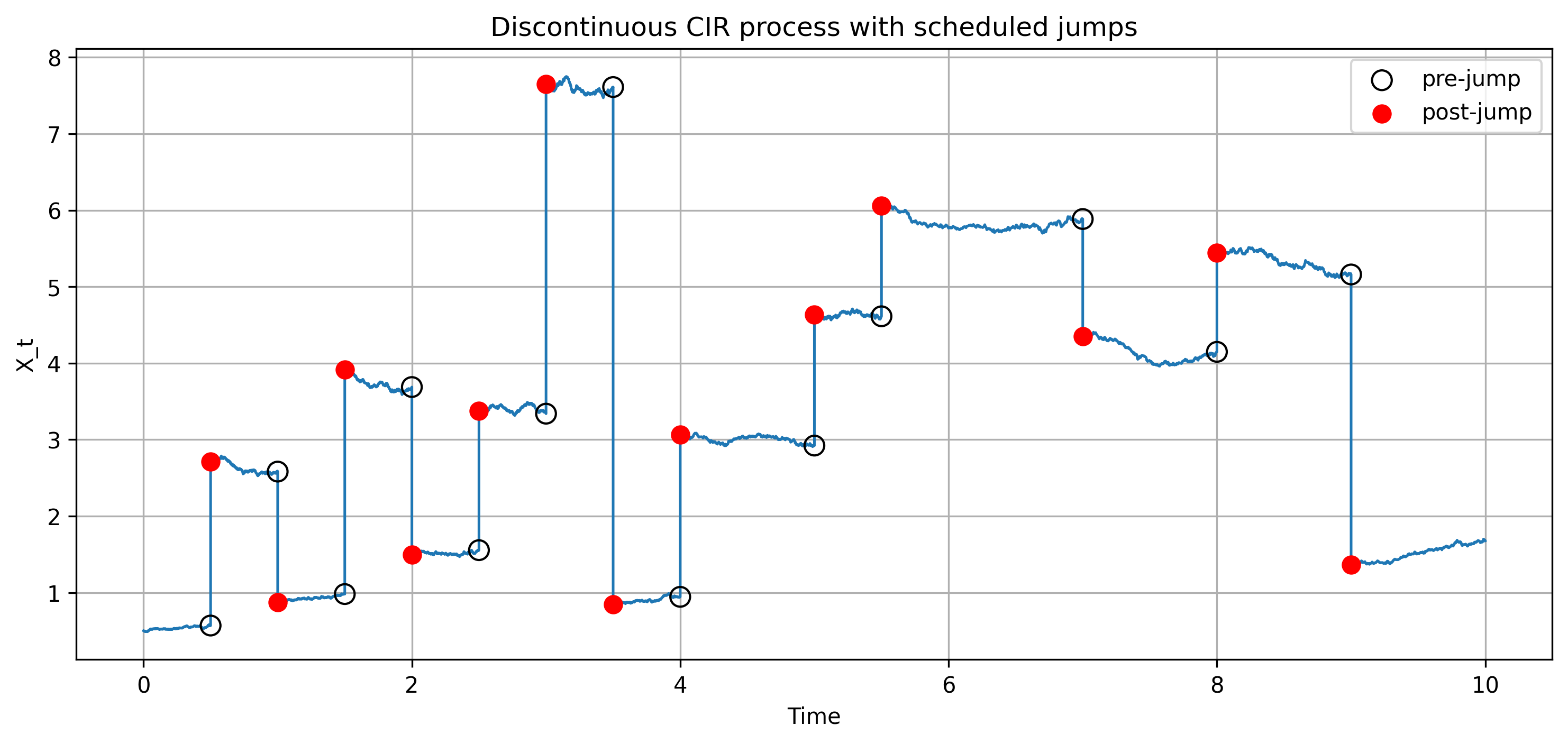}
    \caption{Simulation of a CIR process with stochastic discontinuities following Example~\ref{ex:CIR_Drop_up}, with mean-reversion speed $\kappa = 0.1$, long-term mean $\theta = 3.0$, and volatility $\sigma = 0.1$. The simulation includes $N = 13$ scheduled jumps. Post-jump values $(X_{s_n})_{n=1,\dots,13}$ are Gamma-distributed with shape $\alpha_n + \beta_n X_{s_n-}$: for $n = 1,\dots,10$, parameters are $\alpha_n = 3.0$, $\beta_n = 1.0$, while for $n = 11,12,13$ they are $\alpha_n = 3.5$, $\beta_n = 1.5$. The rate parameter is $\lambda = 1.0$.}
    \label{fig:CIR_Drop_up}
\end{figure}

To introduce the second example, we recall the notion of a time change (cf.~\cite{Jacod1979}). An $\bbF$-time change is a c\`adl\`ag, nondecreasing family $(\tau_t)_{t \geq 0}$ of $\bbF$-stopping times. It is said to be finite if each stopping time is almost surely finite. We denote by $\bbG$ the time-changed filtration, defined by 
\[
\cG_t \coloneqq \cF_{\tau_t}, \qquad t \geq 0.
\]

It is well known that the semimartingale property is preserved under time changes (see Corollary~10.12 in \cite{Jacod1979}). We now show that, under any deterministic time change, not only the semimartingale property but also the affine structure is preserved.  

\begin{proposition}\label{prop:tc_affine}
Let $Y$ be an affine semimartingale taking values in $D\subseteq\R^d$, and let $\tau : \R_+ \to \R_+$ be a finite, deterministic time change. Then the process $X = (X_t)_{t \geq 0}$ defined by 
\[
X_t \coloneqq Y_{\tau(t)}, \qquad t \geq 0,
\]
is an affine semimartingale with respect to the time-changed filtration $\bbG$.
\end{proposition}

\begin{proof}
We verify directly that $X$ is affine by computing its conditional characteristic function.  
For all $0 \leq t \leq T$ and $u \in \mathcal{U}$, we have
\begin{align*}
E \!\left[ e^{\langle u, X_T \rangle} \mid \mathcal{G}_t \right] = E \!\left[ e^{\langle u, Y_{\tau(T)} \rangle} \mid \mathcal{F}_{\tau(t)} \right] = \exp\!\left( \phi_{\tau(t)}(\tau(T), u) + \langle \psi_{\tau(t)}(\tau(T), u), X_t \rangle \right),
\end{align*}
where the last equality follows from the affine property of $Y$. This shows that the time-changed semimartingale $X$ inherits the affine structure with respect to $\bbG$, with characteristic exponents given by
\[
\phi^{\tau}_t(T,u) \coloneqq \phi_{\tau(t)}(\tau(T), u),
\qquad 
\psi^{\tau}_t(T,u) \coloneqq \psi_{\tau(t)}(\tau(T), u),
\]
for all $0 \leq t \leq T$ and $u \in \mathcal{U}$.
\end{proof}

In the next proposition, we illustrate how an extended CIR process with stochastic discontinuities can be constructed by applying a deterministic time change to a standard CIR process. 

\begin{proposition}[Time-changed CIR process]\label{prop:tc_CIR}
Let $Y$ be a continuous CIR process satisfying
\[
dY_t = \kappa (\theta - Y_t) \, dt + \sigma \sqrt{Y_t} \, dW_t, \quad t \ge 0.
\]
Let $\tau : \mathbb{R}_+ \to \mathbb{R}_+$ be a finite, deterministic time change defined by
\[
\tau(t) \coloneqq t + H(t), \qquad 
H(t) \coloneqq \sum_{n:\, s_n \le t} \Delta_n, \quad \Delta_n \ge 0.
\]
Define the time–changed CIR process $X=(X_t)_{t\ge 0}$ by
\[
X_t \coloneqq Y_{\tau(t)}.
\]

Then $X$ is a non-negative affine semimartingale with respect to the time-changed filtration $\mathbb{G}$, satisfying the SDE~\eqref{eq:CIR}, with jump sizes
\[
F_n(x,z) = G_n^{-1}(x,z),
\]
where $G_n(x,y)$ denotes the CDF of a shifted non-central chi-squared random variable $\xi_n$ with conditional density
\[
f_{\xi_n \mid X_{s_n-}=x}(\xi) 
= \frac{1}{c_n}\, f_{\chi^{'2}_{\nu,\lambda_n}}\!\Bigl(\frac{\xi + x}{c_n}\Bigr), 
\qquad \xi > -x,
\]
and
\begin{equation}\label{eq:c_n}
c_n = \frac{\sigma^2 \bigl(1 - e^{-\kappa \Delta_n}\bigr)}{4\kappa},
\end{equation}
with degrees of freedom $\nu$ and non-centrality parameter $\lambda_n$ given by
\begin{equation}\label{eq:lamnda_nu}
\nu = \frac{4 \kappa \theta}{\sigma^2},\quad
\lambda_n = \frac{e^{-\kappa \Delta_n} x}{c_n}.
\end{equation}
\end{proposition}

\begin{proof}
By definition, the time–changed CIR process \(X=(X_t)_{t \geq 0}\) satisfies the following SDE in integral form:
\[
X_t 
= Y_0 + \int_{0}^{\tau(t)} \kappa \bigl(\theta - Y_s\bigr)\, ds 
      + \int_{0}^{\tau(t)} \sigma \sqrt{Y_s}\, dW_s .
\]
Applying the substitution formula for time changes (Proposition~10.21 in \cite{Jacod1979}), this expression can be rewritten as
\begin{equation}\label{eq:CIR_tc}
X_t 
= X_0 + \int_{0}^{t} \kappa \bigl(\theta - X_s\bigr)\, d\tau(s) 
       + \int_{0}^{t} \sigma \sqrt{X_s}\, dW_{\tau(s)} .
\end{equation}
The first integral represents the drift accumulated under the new clock, whose differential is given by
\begin{equation}\label{eq:clock}
d\tau(s) = ds + dH(s).
\end{equation}

The noise term in \eqref{eq:CIR_tc} is a stochastic integral with respect to the time–changed Brownian motion \(W_{\tau(\cdot)}\), which can be decomposed as
\[
W_{\tau(t)} = W_t + U_t ,
\]
where \(U=(U_t)_{t\geq 0}\) is a Gaussian process independent of \(W\), defined by 
\[
U_t \coloneqq W_{\tau(t)} - W_t \sim \cN\bigl(0, H(t)\bigr), 
\qquad t \geq 0.
\]
Using the independent increment property of Brownian motion, we can further decompose \(U_t\) as
\[
U_t = \sum_{n:\, s_n \leq t} \bigl(W_{t+H(s_n)} - W_{t+H(s_n)-\Delta_n}\bigr), 
\qquad t \geq 0.
\]
Hence, the time–changed Brownian motion admits the representation
\[
W_{\tau(t)} = W_t + \sum_{n:\, s_n \leq t} \eta_n ,
\]
where \((\eta_n)_{n \geq 1}\) is a sequence of independent Gaussian random variables with
\[
\eta_n \sim \cN(0,\Delta_n).
\]
In particular, \(W_{\tau(\cdot)}\) is a semimartingale with continuous martingale part \(W\), and deterministic jump times \(s_n\) carrying independent Gaussian jumps of mean zero and variance \(\Delta_n\).

By substituting the decomposition of \(W_{\tau(\cdot)}\) together with \eqref{eq:clock} into \eqref{eq:CIR_tc}, we obtain that, under the time–changed filtration \(\mathbb{G}\), the process \(X\) satisfies the stochastic dynamics
\[
X_t 
= X_0 + \int_{0}^{t} \kappa \bigl(\theta - X_s\bigr)\, ds 
       + \int_{0}^{t} \sigma \sqrt{X_s}\, dW_s
       + \sum_{n:\, s_n \leq t} \xi_n ,
\]
where the jump sizes \((\xi_n)_{n \geq 1}\) are given by
\[
\xi_n = \int_{0}^{t} \Big( \kappa(\theta - X_s)\Delta_n 
              + \sigma \sqrt{X_s}\,\eta_n \Big)\, \delta_{s_n}(ds).
\]
Equivalently, the jump sizes can be expressed as
\[
\xi_n = X_{s_n} - X_{s_n-} 
      = Y_{\tau(s_n)} - Y_{\tau(s_n-)} 
      = Y_{t_n + \Delta_n} - Y_{t_n},
\]
where
\[
t_n \coloneqq \tau(s_n-) = s_n + H(s_n) - \Delta_n .
\]
In particular, for any measurable set \(A \in \cB(\R)\),
\[
P\!\left(\xi_n \in A \mid \mathcal{G}_{s_n-}\right) 
= P\!\left(Y_{t_n+\Delta_n} - Y_{t_n} \in A \mid \mathcal{F}_{t_n}\right)= P\!\left(Y_{t_n+\Delta_n} - Y_{t_n} \in A \mid Y_{t_n}\right).
\]
By the distributional properties of the CIR process (see, e.g., \cite[Proposition~1.2.4]{Alfonsi2015Affine}), the conditional distribution of \(\xi_n\) given \(X_{s_n-} = Y_{t_n} = x\) coincides with that of
\[
c_n V_n - x,
\]
where \(c_n\) is defined in \eqref{eq:c_n}, and \(V_n\) is a non-central chi–squared random variable with parameters \(\nu\) and \(\lambda_n\) as in \eqref{eq:lamnda_nu}.

By choosing \(F_n\) as in the statement, it follows that \(\xi_n = F_n(X_{s_n},Z_n)\), so that \(X\) is an extended CIR process satisfying~\eqref{eq:CIR}. In particular, the conditional support of \(\xi_n\) given the pre-jump state \(X_{s_n-}=x\) is \([-x, \infty)\), which guarantees that the admissibility condition in Definition~\ref{def:support} is fulfilled.
\end{proof}

\begin{remark}
It is worth noting that a time–changed CIR process is an extended CIR process whose jump sizes coincide with the increments of the original CIR process. Consequently, the jump sizes are autocorrelated, since the mean–reverting nature of the CIR process induces dependence between its increments. 

In particular, using the same notation as in Proposition~\ref{prop:cov_jump}, we obtain
\begin{align*}
c(n,m) \;&=\; \cov\!\left(Y_{t_n+\Delta_n} - Y_{t_n}, \, Y_{t_m+\Delta_m} - Y_{t_m}\right) ) \\[0.5em]
&=\; e^{\kappa\abs{t_n+\Delta_n - t_m - \Delta_m}} 
   \, \var\!\left(Y_{(t_n+\Delta_n)\wedge (t_m+\Delta_m)}\right) - e^{\kappa\abs{t_n+\Delta_n - t_m}}
   \, \var\!\left(Y_{(t_n+\Delta_n)\wedge t_m}\right) \\[0.5em]
&\quad - e^{\kappa\abs{t_m+\Delta_m - t_n}}
   \, \var\!\left(Y_{(t_m+\Delta_m)\wedge t_n}\right) + e^{\kappa\abs{t_n - t_m}}
   \, \var\!\left(Y_{t_n\wedge t_m}\right),
\end{align*}
where the second equality follows from linearity of covariance together with the covariance structure of the CIR process. The explicit form of the variances is known and can then be computed using the formula for the variance of the non–central chi–squared distribution.
\end{remark}

A direct consequence of Proposition~\ref{prop:tc_CIR} is that an extended CIR process has jump sizes with conditional characteristic function
\[
\mathbb{E} \!\left[ e^{u F_n(X_{s_n-},Z_n)} \,\middle|\, \mathcal{G}_{s_n-} \right]
= (1-2uc_n)^{-\frac{\nu}{2}}
   \exp\!\left(
      \frac{u e^{-\kappa \Delta_n} X_{s_n-}}{1-2uc_n} - u X_{s_n-}
   \right),
\]
which is of the affine form given in~\eqref{eq:affine_jump}, with
\begin{align*}
\gamma_{n,0}(u) &= -\tfrac{\nu}{2}\log(1-2uc_n), \\
\gamma_{n,1}(u) &= u \left( \frac{e^{-\kappa \Delta_n}}{1-2uc_n} - 1 \right).
\end{align*}

Notably, we observe the following asymptotic behavior:  
\begin{align*}
\gamma_{n,0}(u) &\to 0, \quad \gamma_{n,1}(u) \to 0, \quad \text{as } \Delta_n \to 0,\\
\gamma_{n,0}(u) &\to -\tfrac{\nu}{2} \log\!\left(1 - u \frac{\sigma^2}{2\kappa}\right), \quad 
\gamma_{n,1}(u) \to -u, \quad \text{as } \Delta_n \to \infty.
\end{align*}
In particular, as $\Delta_n \to 0$, the corresponding jump size $\xi_n$ converges to $0$ in distribution, implying that the post-jump values $X_{s_n}$ approach the pre-jump states $X_{s_n-}$. Conversely, as $\Delta_n \to \infty$, the conditional characteristic function of the post-jump values satisfies
\begin{align*}
\lim_{\Delta_n \to \infty} E \left[ e^{u X_{s_n}} \,\middle|\, \mathcal{G}_{s_n-} \right]
= \lim_{\Delta_n \to \infty} e^{u X_{s_n-}} \, E \left[ e^{u \xi_n} \,\middle|\, \mathcal{G}_{s_n-} \right]
= \left(1 - u \frac{\sigma^2}{2\kappa}\right)^{-\frac{\nu}{2}}.
\end{align*}
Hence, as $\Delta_n \to \infty$, the distribution of $X_{s_n}$ converges to a Gamma distribution with parameters 
\(\alpha = \frac{2\kappa \theta}{\sigma^2}\) and \(\lambda = \frac{2\kappa}{\sigma^2}\), corresponding to the stationary distribution of the CIR process.

A simulation of a CIR process with stochastic discontinuities constructed via time change is shown in Figure~\ref{fig:CIR_Drop_up}.

\begin{figure}[t]
    \centering
    \includegraphics[width=0.9\textwidth]{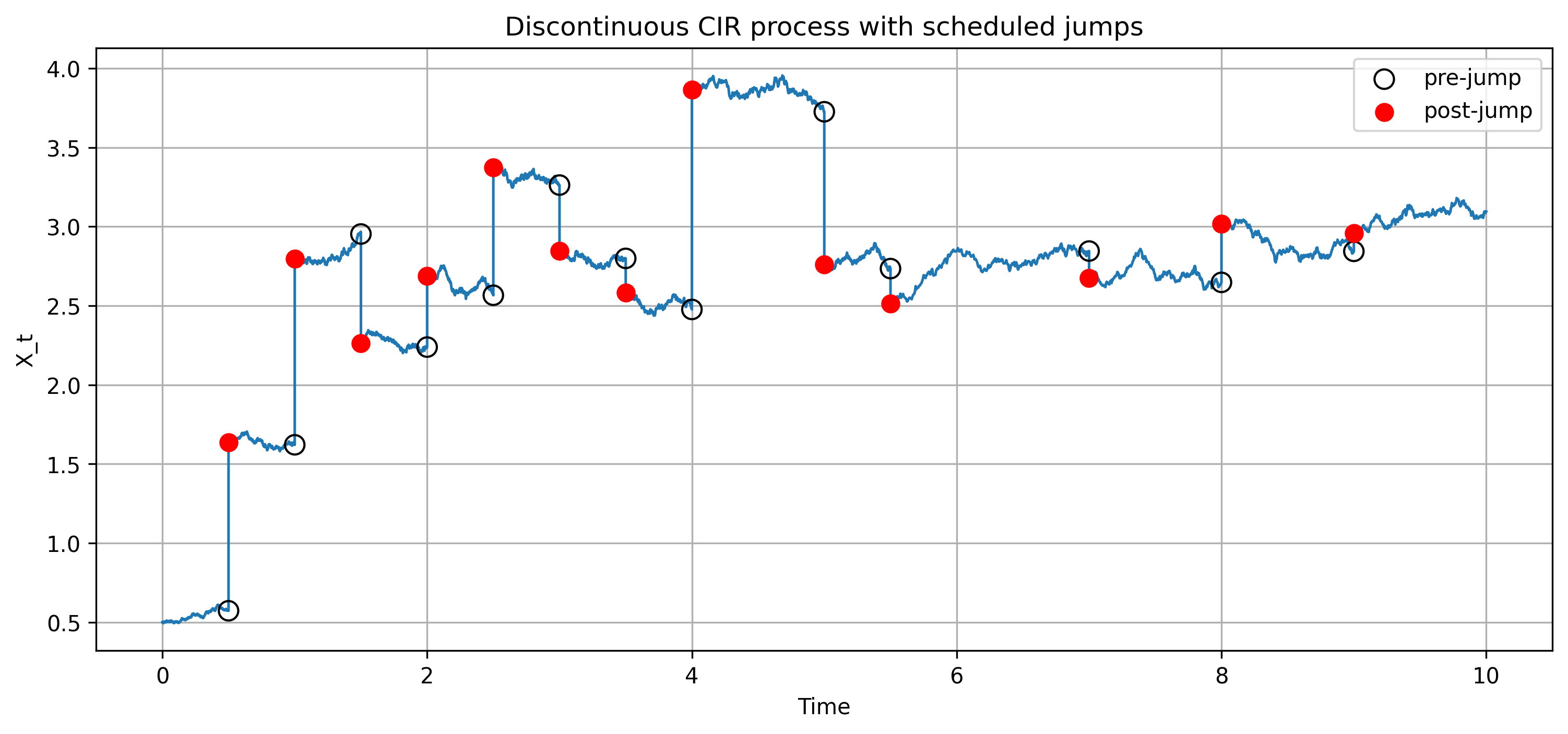}
    \caption{Simulation of an extended CIR process obtained via time change, with mean-reversion rate $\kappa = 0.1$, long-term mean $\theta = 3.0$, and volatility $\sigma = 0.1$. The simulation includes $N = 13$ scheduled jumps. The time-change induces jumps such that the post-jump values $(X_{s_n})_{n=1,\dots,13}$ are given by $X_{s_n} = c_n V_n$, where $c_n$ is defined in Equation~\eqref{eq:c_n} and $V_n$ is non-central chi-squared distributed with degrees of freedom and non-centrality parameter given by Equation~\eqref{eq:lamnda_nu}. The length of the time-change shifts $\Delta_n$ is $15$ for $n=1,\dots,5$, $20$ for $n=6,\dots,10$, and $25$ for $n=11,12,13$.}
    \label{fig:CIR_tc}
\end{figure}

\subsection{Necessary conditions for the affine property}

In Theorem~\ref{th:suff}, we established sufficient conditions for the extended CIR process to be a quasiregular affine semimartingale. We now extend this result by incorporating necessary conditions as well. To this end, recall that an affine semimartingale \(X=(X_t)_{t\geq 0}\) taking values in a closed convex cone of full dimension \(D \subseteq \R^d\) is said to have \emph{support of full convex span} if 
\begin{equation}\label{eq:full_conv}
\operatorname{conv}(\operatorname{supp}(X_t)) = D, \quad \text{for all } t \geq 0.
\end{equation}
This condition is minimal and underlies most structural properties of affine semimartingales (see \cite{keller-ressel2019}).  

In our setting, Condition~\eqref{eq:full_conv} with $D=\R_+$ is satisfied for all \(t \notin \cS\), since between jumps the process evolves as a continuous CIR diffusion, whose support is \(\R_+\). The following lemma characterizes the full convex span condition for the support of the post-jump values $(X_{s_n})_{n\geq 1}$.

\begin{lemma}\label{lem:jump_conv}
Assume that the functions \(F_n:\R_+ \times [0,1] \to \R\), \(n \geq 1\), are continuous, and set
\[
g_n(x,z) \coloneqq x + F_n(x,z).
\]
Then,
\[
\operatorname{conv}(\operatorname{supp}(X_{s_n})) = \R_+ \quad \text{for all } n \geq 1
\]
if and only if
\[
\inf_{(x,z) \in \R_+ \times [0,1]} g_n(x,z) = 0
\quad \text{and} \quad 
\sup_{(x,z) \in \R_+ \times [0,1]} g_n(x,z) = \infty,
\]
for all $n\geq 1$. 
\end{lemma}

\begin{proof}
Since $F_n$ is continuous for all $n \geq 1$, it follows that $g_n$ is also continuous. By Lemma~\ref{lem:supp_continuous_map}, we obtain
\[
\operatorname{supp}(X_{s_n})
= \operatorname{cl}\big(g_n(\operatorname{supp}(X_{s_n-},Z_n))\big)
= \operatorname{cl}\big(g_n(\R_+ \times [0,1])\big),
\]
where the last equality uses that $X_{s_n-}$ and $Z_n$ are independent for all $n \geq 1$, and hence
\[
\operatorname{supp}(X_{s_n-},Z_n)
= \operatorname{supp}(X_{s_n-}) \times \operatorname{supp}(Z_n)
= \R_+ \times [0,1].
\]
In particular, we may write
\begin{equation}\label{eq:supp_jump}
\operatorname{supp}(X_{s_n})
= \operatorname{cl}\Bigg(\bigcup_{(x,z)\in\R_+ \times [0,1]} g_n(x,z)\Bigg).
\end{equation}

Now assume that $\operatorname{conv}(\operatorname{supp}(X_{s_n})) = \R_+$ for all $n \geq 1$. This is equivalent to  
\[
\operatorname{supp}(X_{s_n}) \subseteq \R_+,\quad \text{with} \quad \inf \operatorname{supp}(X_{s_n}) = 0, \,\sup \operatorname{supp}(X_{s_n}) = \infty.
\]

Using Equation~\eqref{eq:supp_jump}, the condition on the infimum can be restated as follows:
\begin{enumerate}
    \item $g_n(x,z) \geq 0$ for all $(x,z) \in \R_+ \times [0,1]$;
    \item for every $\varepsilon > 0$, there exists $(x^*, z^*) \in \R_+ \times [0,1]$ such that 
    \[
    g_n(x^*, z^*) < \varepsilon.
    \]
\end{enumerate}

Combining these two conditions yields
\[
0 \leq \inf_{(x,z)\in\R_+\times [0,1]} g_n(x,z) \leq g_n(x^*,z^*) < \varepsilon.
\]
Letting $\varepsilon \to 0$, we conclude that 
\[
\inf_{(x,z) \in \R_+ \times [0,1]} g_n(x,z) = 0.
\]

A similar argument applied to the supremum of $\operatorname{supp}(X_{s_n})$ shows that $g_n$ is unbounded over its domain.

Suppose now that 
\[
\inf_{(x,z) \in \R_+ \times [0,1]} g_n(x,z) = 0
\quad \text{and} \quad 
\sup_{(x,z) \in \R_+ \times [0,1]} g_n(x,z) = \infty.
\]
The condition on the infimum can be rewritten as
\[
\inf g_n(\R_+ \times [0,1]) = 0,
\]
which immediately implies
\[
\inf \operatorname{supp}(X_{s_n}) = \inf \operatorname{cl}(g_n(\R_+ \times [0,1])) = 0.
\]
Similarly, the condition on the supremum of $g_n$ ensures that $\operatorname{supp}(X_{s_n})$ is unbounded above. Together, these two conditions are equivalent to $\operatorname{conv}(\operatorname{supp}(X_{s_n})) = \R_+$.
\end{proof}

Before turning to the main result of this section, we illustrate the support condition with two simple examples: one in which the full convex span holds, and one in which it fails.  

\begin{example}[Trivial case without jumps]
If $F_n(x,z)=0$ for all $(x,z)\in\R_+\times [0,1]$ and $n\geq 1$, then 
\[
\operatorname{supp}(X_{s_n})=\operatorname{supp}(X_{s_n-})=\R_+,
\]
since no jump occurs and the process reduces to a continuous CIR model.  

Hence, the full convex span condition is satisfied, and it is immediate to check that the conditions of Lemma \ref{lem:jump_conv} are fulfilled.  
\end{example}

\begin{example}[Deterministic unit upward jumps]
If $F_n(x,z)=1$ for all $(x,z)\in\R_+\times [0,1]$ and $n\geq 1$, then 
\[
\operatorname{supp}(X_{s_n})=\operatorname{supp}(X_{s_n-}+1)=\R_+ + \{1\},
\] 
where 
\[
\R_+ + \{1\}=\{x+1 : x\in\R_+\}.
\]
Hence, 
\[
\operatorname{supp}(X_{s_n})=[1,\infty),
\]
which is convex and coincides with its convex hull, but does not equal $\R_+$, so the full convex span condition is violated.  

Indeed, 
\[
\inf_{(x,z) \in \R_+ \times [0,1]} g_n(x,z)=\inf_{x\in \R_+} (x+1) = 1\neq 0. 
\]
Thus, in the presence of strictly upward jumps, the full convex span condition fails whenever the jump distribution admits a strictly positive lower bound.  
\end{example}

We now combine the previous results to establish the necessary and sufficient conditions for the affine property of the extended CIR process.

\begin{theorem}\label{th:suff_nec}
Suppose that Assumption~\ref{ass:accumulation} holds and that the functions \(F_n:\R_+ \times [0,1] \to \R\), \(n \geq 1\), are continuous. Then there exists a unique quasiregular affine semimartingale \(X\) solving the SDE~\eqref{eq:CIR} such that
\[
\operatorname{conv}(\operatorname{supp}(X_t)) = \R_+ \quad \text{for all } t \geq 0,
\]
if and only if the jump sizes $(F_n)_{\geq 1}$ admit a conditional characteristic function of the form
\[
E \left[ e^{u F_n(X_{s_n-},Z_n)} \,\middle|\, \mathcal{F}_{s_n-} \right]
  = \exp\bigl( \gamma_{n,0}(u) + \gamma_{n,1}(u)\,X_{s_n-} \bigr),
  \qquad u\in i\mathbb{R},
\]
for some continuous functions \(\gamma_{n,0},\gamma_{n,1} : i\mathbb{R}\to\mathbb{C}\), and
\[
\inf_{(x,z) \in \R_+ \times [0,1]} \bigl(x+F_n(x,z)\bigr) = 0,
\qquad
\sup_{(x,z) \in \R_+ \times [0,1]} \bigl(x+F_n(x,z)\bigr) = \infty,
\]
for all $n\geq 1$. 
\end{theorem}

\begin{proof}
Assume there exists a unique quasiregular affine semimartingale \(X=(X_t)_{t\geq 0}\) solving \eqref{eq:CIR} and satisfying
\[
\operatorname{conv}(\operatorname{supp}(X_t)) = \R_+, \qquad t \ge 0.
\]
By Lemma~2.10 in \cite{keller-ressel2019}, for each deterministic jump time \(s_n\) the conditional characteristic function of the jump
\(F_n(X_{s_n-},Z_n)\) admits the representation
\begin{align*}
E\!\left[ e^{u F_n(X_{s_n-},Z_n)} \mid \mathcal{F}_{s_n-} \right]
&= \exp\!\big( -\Delta\phi_{s_n}(s_n,u) - \Delta\psi_{s_n}(s_n,u)\,X_{s_n-} \big) \\
&= \exp\!\big( \phi_{s_n-}(s_n,u) + (\psi_{s_n-}(s_n,u)-u)\,X_{s_n-} \big),
\end{align*}
for \(u\in i\R\), where \(\phi\) and \(\psi\) are the characteristic exponents of the affine semimartingale \(X\).

Thus the jump sizes \((F_n(X_{s_n-},Z_n))_{n\ge 1}\) have a conditional characteristic function of affine form, with
\begin{equation}\label{eq:gamma}
\gamma_{n,0}(u)\coloneqq \phi_{s_n-}(s_n,u), \qquad
\gamma_{n,1}(u)\coloneqq \psi_{s_n-}(s_n,u)-u.
\end{equation}
By quasiregularity, the maps \(\gamma_{n,0}\) and \(\gamma_{n,1}\) are continuous on \(i\R\). The conditions on the infimum and supremum of \(x + F_n(x,z)\) follow directly from Lemma~\ref{lem:jump_conv}.

For the converse implication, note that the infimum condition of \(x + F_n(x,z)\) implies
\[
0 \leq x + \inf_{z \in [0,1]} F_n(x,z) 
   \leq x + \essinf_{z \in [0,1]} F_n(x,z),
   \quad \text{for all } x \in \R_+,\, n \geq 1,
\]
since, on a space of positive measure, the essential infimum is never smaller than the pointwise infimum.  

Hence, this condition is stronger than the admissibility requirement stated in Definition~\ref{def:support}. The result then follows directly by combining Lemmas~\ref{lem:affine_supp}, \ref{lem:jump_conv}, and Theorem~\ref{th:suff}.
\end{proof}

\begin{remark}
By Lemma~2.7 in \cite{keller-ressel2019}, both mappings 
\[
u \mapsto \phi_{s_n-}(s_n,u) \quad\text{and}\quad u \mapsto \psi_{s_n-}(s_n,u)
\] 
send 
\( \mathbb{C}_{-} \) into \( \mathbb{C}_{-} \). Consequently, the functions 
\( \gamma_{n,0} \) and \( \gamma_{n,1} \) are well defined on the left half-plane, 
consistent with Lemma~\ref{lem:affine_supp}. 

Moreover, our results can be extended to time-dependent state spaces. In this more general setting, the support of the jump 
distribution may be smaller: it suffices that the support contains an open ball of full dimension, which in turn guarantees that its linear hull has full dimension.
\end{remark}

\section{Infinite divisibility of extended CIR processes}\label{sec:3}

Theorem~\ref{th:suff_nec} provides a complete characterization of the extended CIR process as an affine semimartingale. In this section, we further develop the affine framework by investigating the conditions under which the extended CIR process is infinitely divisible.

First, recall that to every affine semimartingale $X=(X_t)_{t\geq 0}$ on a state space \( D \subseteq \mathbb{R}^d \), one can associate a family of transition kernels \( p_{t,T}(x,B) \), defined for all \( 0 \leq t \leq T \), \( B \in \mathcal{B}(D) \), and \( x \in \operatorname{supp}(X_t) \), by considering the regular conditional distributions
\[
p_{t,T}(X_t, B) \coloneqq P(X_T \in B \mid X_t).
\]
An affine semimartingale \( X \) on \( D \) is said to be \emph{infinitely divisible} if, for every \( 0 \leq t \leq T \), the regular conditional distributions \( p_{t,T}(X_t, \cdot) \) are infinitely divisible probability measures on \( D \), \( P \)-almost surely (cf. Definition 4.2 in \cite{keller-ressel2019}).

Since the continuous CIR process is infinitely divisible, it is natural to expect that this property carries over to its extended version, provided that the jump sizes themselves are infinitely divisible, which is equivalent to their conditional characteristic functions admitting a Lévy--Khintchine representation.  

As in the general affine framework, we first characterize the admissibility condition given in Definition~\ref{def:support} for the case of infinitely divisible jump distributions. Since infinite divisibility is stronger than the affine property, the following result can be deduced from Lemma~\ref{lem:affine_supp} (cf. Remark~\ref{rem:inf_div} below). Nevertheless, we provide an independent proof.

\begin{lemma}\label{lem:support_IDD}
Suppose that Assumption~\ref{ass:accumulation} holds and that the jump sizes \( (F_n)_{n \geq 1} \) have a conditional characteristic function of the form  
\[
\mathbb{E} \left[ e^{u F_n(X_{s_n-},Z_n)} \mid \mathcal{F}_{s_n-} \right] = \exp \left( \gamma_{n,0}(u) + \gamma_{n,1}(u) X_{s_n-} \right), \quad \text{for all } u \in i \mathbb{R},
\]
where the functions \( \gamma_{n,0} \) and \( \gamma_{n,1} \) admit a Lévy--Khintchine representation:
\begin{align}\label{eq:general_LK}
\gamma_{n,j}(u) 
= u \beta_{n,j} + \frac{1}{2} \alpha_{n,j}^2 u^2 
+ \int_{\mathbb{R} \setminus \{0\}} \left( e^{u \xi} - 1 - u \xi \Ind_{\{ |\xi| \leq 1 \}} \right) \nu_{n,j}(d\xi),
\quad j \in \{0,1\},
\end{align}
with \( \beta_{n,j} \in \mathbb{R} \), \( \alpha_{n,j} \geq 0 \), and \( \nu_{n,j} \) being Borel measures satisfying
\[
\int_{\mathbb{R} \setminus \{0\}} \min(1, \xi^2) \, \nu_{n,j}(d\xi) < \infty, \quad j \in \{0,1\}.
\]

Then the sequence \( (F_n)_{n\geq 1} \) is admissible if and only if the following conditions hold for all \( n \geq 1 \):
\begin{itemize}
    \item[(i)] The diffusion coefficients vanish:
    \[
    \alpha_{n,0} = \alpha_{n,1} = 0.
    \]
    \item[(ii)] The drift coefficients satisfy:
    \[
    \beta_{n,0} \geq 0, \quad \beta_{n,1} \geq -1.
    \]
    \item[(iii)] The Lévy measures \( \nu_{n,0} \) and \( \nu_{n,1} \) are supported on \( \R_+\) and satisfy
    \[
    \int_0^1 \xi \, \nu_{n,j}(d\xi) < \infty, \quad j \in \{0,1\}.
    \]
\end{itemize}
\end{lemma}

\begin{proof}
As in the proof of Lemma~\ref{lem:affine_supp}, the admissibility of the sequence \( (F_n)_{n\geq 1} \) implies that the distribution of \( F_n(x, Z_n) \) is supported on \( [-x, \infty) \) for all \( x \in \mathbb{R}_+ \).

By Theorem~24.7 in \cite{Sato} (see also \cite[Theorem~11.2.2]{lukacs1970characteristic}), the conditional distribution of each jump size is bounded from below if and only if the diffusion coefficients \( \alpha_{n,j} \) in the Lévy--Khintchine representation vanish and the Lévy measures \( \nu_{n,j} \) are supported on \( \R_+ \), satisfying the integrability condition stated in~(iii).

Moreover, by Corollary~24.8 of \cite{Sato}, the infimum of the support of \( F_n(x,Z_n) \) is given by
\[
\inf \operatorname{supp}(F_n(x,Z_n)) = \beta_{n,0} + \beta_{n,1} x, \quad \text{for all } x \in \mathbb{R}_+.
\]

Using the identity
\[
\inf \operatorname{supp} F_n(x,Z_n) = \essinf_{z \in [0,1]} F_n(x,z),
\]
and substituting the above expression into the admissibility condition from Definition~\ref{def:support}, we obtain
\[
\beta_{n,0} + x(\beta_{n,1} + 1) \geq 0, \quad \text{for all } x \in \mathbb{R}_+.
\]
Since this equality holds for all \( x \in \mathbb{R}_+ \), it directly yields Condition~(ii) on the drift coefficients.

The converse implication follows by reversing the previous steps.
\end{proof}
     
\begin{remark}\label{rem:inf_div} 
In view of the integrability condition in~(iii) and by Remark~8.4 in \cite{Sato}, we obtain
\[
\int_0^{\infty} \bigl( e^{u \xi} - 1 \bigr) \, \nu_{n,j}(d\xi) < \infty,
\quad \text{for all } n \geq 1 \text{ and } j \in \{0,1\}.
\]
Hence, a sequence of jump sizes $(F_n)_{n\geq 1}$ with conditionally infinitely divisible distribution is admissible if and only if the functions \( \gamma_{n,0} \) and \( \gamma_{n,1} \) admit a Lévy--Khintchine representation of the form 
\begin{align*}
\gamma_{n,j}(u) 
= u \beta_{n,j} + \int_0^{\infty} \bigl( e^{u \xi} - 1 \bigr) \nu_{n,j}(d\xi),
\quad u \in i \mathbb{R},\ j \in \{0,1\},
\end{align*}
with $\beta_{n,0}\geq 0$ and $\beta_{n,1}\geq 1$.

The drift restriction can equivalently be derived from Condition~(ii) of Lemma~\ref{lem:affine_supp}, which in the infinitely divisible setting yields
\begin{align*}
\lim_{y \to \infty} \frac{\gamma_{n,0}(-y)}{y} 
&= -\beta_{n,0} + \lim_{y \to \infty} \int_0^\infty \frac{e^{-y \xi} - 1}{y} \, \nu_{n,0}(d\xi) \leq 0, \\
\lim_{y \to \infty} \frac{\gamma_{n,1}(-y)}{y} 
&= -\beta_{n,1} + \lim_{y \to \infty} \int_0^\infty \frac{e^{-y \xi} - 1}{y} \, \nu_{n,1}(d\xi) \leq 1.
\end{align*}
By dominated convergence, the integral terms vanish in the limit, so that Condition~(ii) of Lemma~\ref{lem:support_IDD} is recovered.

Moreover, taking real parts of $\gamma_{n,0}$ and $\gamma_{n,1}$ gives
\[
\Re(\gamma_{n,j}(u))
= \Re(u)\beta_{n,j}
+ \int_0^{\infty} \bigl( e^{\Re(u)\xi}\cos(\Re(u)\xi) - 1 \bigr) \nu_{n,j}(d\xi),
\quad j \in \{0,1\}.
\]
By Condition~(ii) of Lemma~\ref{lem:support_IDD}, it follows that for all $u \in \C_{-}$,
\[
\Re(\gamma_{n,0}(u)) \leq 0,
\qquad
\Re(\gamma_{n,1}(u)) \leq |u|,
\]
which in turn implies that $\gamma_{n,0}$ and $\gamma_{n,1}$ satisfy Condition~(i) of Lemma~\ref{lem:affine_supp} with $C(x)=x$.
\end{remark}

We now establish sufficient conditions under which the extended CIR process defines an infinitely divisible affine semimartingale.

\begin{theorem}\label{th:inf_div_dist}
Suppose that Assumption~\ref{ass:accumulation} holds, and that the jump sizes $(F_n)_{n\geq 1}$ admit a conditional characteristic function of Lévy--Khintchine form
\begin{equation}\label{eq:Levy}
\mathbb{E} \left[ e^{u F_n(X_{s_n-},Z_n)} \mid \mathcal{F}_{s_n-} \right] 
= \exp \left( \beta_n(X_{s_n-})u + \int_0^\infty \left( e^{u \xi} - 1 \right)\, \nu_n(d\xi,X_{s_n-}) \right),
\end{equation}
where
\[
\beta_n(x) = \beta_{n,0} + \beta_{n,1}x,
\qquad 
\nu_n(d\xi,x) = \nu_{n,0}(d\xi) + \nu_{n,1}(d\xi)\,x,
\quad x \in \R_+,
\]
with $\beta_{n,0} \geq 0$, $\beta_{n,1} \geq -1$, and where $\nu_{n,0}$ and $\nu_{n,1}$ are Borel measures supported on $\R_+$ satisfying
\[
\int_0^1 \xi \, \nu_{n,j}(d\xi) < \infty, 
\qquad j \in \{0,1\}.
\]
Then there exists a unique infinitely divisible, quasi-regular, and almost surely non-negative affine semimartingale \( X \) solving the SDE~\eqref{eq:CIR}.
\end{theorem}

\begin{proof} 
If the conditional characteristic functions of the jump sizes \( (F_n)_{n\geq 1} \) admit the Lévy--Khintchine representation~\eqref{eq:Levy}, then in particular they are of the affine exponential form~\eqref{eq:affine_jump}. Moreover, by Lemma~\ref{lem:support_IDD}, the sequence \( (F_n)_{n\geq 1} \) is admissible in the sense of Definition~\ref{def:support}. 

It then follows from Lemma~\ref{lem:affine_supp} that the functions \( \gamma_{n,0} \) and \( \gamma_{n,1} \), defined by
\[
\gamma_{n,j}(u) 
= u \beta_{n,j} + \int_0^{\infty} \bigl( e^{u \xi} - 1 \bigr) \nu_{n,j}(d\xi),
\quad u \in i \mathbb{R},\ j \in \{0,1\},
\]
with parameters \( \beta_{n,j} \) and \( \nu_{n,j} \) satisfying the conditions stated in the theorem, fulfill Conditions~(i)--(ii) of Lemma~\ref{lem:affine_supp}.

Applying Theorem~\ref{th:suff}, we then conclude that there exists a unique quasi-regular, almost surely non-negative affine semimartingale \( X \) solving the SDE~\eqref{eq:CIR}.

To establish infinite divisibility, we first observe that the good parameter set (see Definition~3.1 of \cite{keller-ressel2019}) corresponding to the semimartingale \( X \) is given by
\begin{equation}\label{eq:param}
(A, \gamma, \widehat\Theta, \widehat\Sigma) = (\; A, \gamma_j, \widehat\Theta_j, \widehat\Sigma_j \;)_{j=0,1}.
\end{equation}

Here, \( A:\R_+ \to \R_+ \) is the non-decreasing, càdlàg function
\[
A(t) = t + N(t), \quad t \geq 0,
\]

where \( N \) denotes the counting function defined in Proposition~\ref{prop:compensator}, and \( \gamma_j : \R_+ \times \C_- \to \C^2 \) are complex-valued functions defined by
\begin{align*}
\gamma_j(t,u) &= 0, \quad \text{for all } (t,u) \in (\R_+ \setminus \cS) \times \C_-,\\
\gamma_j(s_n,u) &= \gamma_{n,j}(u), \quad \text{for all } (s_n,u) \in \cS \times \C_-,
\end{align*}
for \( j \in \{0,1\} \). Finally, the vectors \( \widehat\Theta, \widehat\Sigma \in \R^2 \) are given by
\[
\widehat\Sigma_0 = 0, \quad \widehat\Sigma_1 = \sigma, \quad \widehat\Theta_0 = \kappa \theta, \quad \widehat\Theta_1 = -\kappa.
\]

Since \( \Delta A_{s_n} = 1 \) for all \( n \geq 1 \), following the procedure in \cite[Definition~4.5]{keller-ressel2019}, the set of parameters in \eqref{eq:param} can be enhanced by defining the functions \( \Theta, \Sigma : \R_+ \to \R^2 \) as follows:

\begin{align*}
\Theta_0(t) &= \kappa \theta, \quad \text{for all } t \in \R_+ \setminus \cS, 
\qquad \Theta_0(s_n) = \beta_{n,0}, \quad \text{for all } s_n \in \cS, \\
\Theta_1(t) &= -\kappa, \quad \text{for all } t \in \R_+ \setminus \cS, 
\qquad \Theta_1(s_n) = \beta_{n,1}, \quad \text{for all } s_n \in \cS,
\end{align*}

and
\[
\Sigma_0(t) = 0, \quad \text{for all } t \geq 0, \qquad 
\Sigma_1(t) = \sigma \, \Ind_{t \in \R_+ \setminus \cS}.
\]
By also incorporating the Lévy measures, we define the enhanced set of parameters of \( X \) as
\begin{equation}\label{eq:param_enh}
(A, \Theta, \Sigma, \nu) = (\; A, \Theta_j, \Sigma_j, \nu_j \;)_{j=0,1}.
\end{equation}

We now verify that this enhanced set of parameters is admissible in the sense of Conditions~(i)--(ii') of Definition~5.1 in \cite{keller-ressel2019}. Since the state space of \( X \) is \( \R_+ \), the admissibility conditions to be checked are those corresponding to the non-negative components of affine semimartingales in \( \R^d \). Condition~(i) of Definition~5.1, which concerns the continuous part of the affine semimartingale, is immediately satisfied for the extended CIR process.

Condition~(ii') follows directly by construction. At the jump dates \( (s_n)_{n\geq 1} \), the diffusion coefficients \(
\Sigma_0(s_n)\) and \(\Sigma_1(s_n)\) vanish, the drift parameters satisfy
\[
\Theta_0(s_n) = \beta_{n,0} \geq 0, \quad \Theta_1(s_n) = \beta_{n,1} \geq -1,
\] 
and the Lévy measures \( \nu_{n,0} \) and \( \nu_{n,1} \) are supported on \( \R_+ \), with the required integrability conditions holding.

Hence, the enhanced parameter family \eqref{eq:param_enh} satisfies Conditions~(i)--(ii') of Definition~5.1 in \cite{keller-ressel2019}. 

Applying Theorem~5.7 of \cite{keller-ressel2019}, we conclude that there exists an infinitely divisible affine semimartingale solving the SDE~\eqref{eq:CIR} with almost surely non-negative paths. By uniqueness, this process coincides with \( X \), and is therefore also quasi-regular.
\end{proof}

\begin{remark}
In view of \cite[Lemma~4.4]{keller-ressel2019}, a complete characterization of infinitely divisible extended CIR processes is possible. Specifically, Theorem~\ref{th:inf_div_dist} can be extended to include the converse implication. For this purpose, the same conditions as in Theorem~\ref{th:suff_nec} on the range of the functions \( x + F_n(x,z) \), which generate the jump states \( (X_{s_n})_{n \geq 1} \), must be imposed to ensure the good behaviour of \( \operatorname{supp}(X_{s_n-}) \).

It is also worth noting that both examples presented in Section~\ref{sec:1} fall into the class of infinitely divisible extended CIR processes, since their jump-size distributions are infinitely divisible. In particular, for the time-changed CIR model, infinite divisibility is preserved under the specified time change, as expected from \cite[Theorem~30.1]{Sato}.
\end{remark}

\appendix

\section{Auxiliary results}

This appendix is devoted to auxiliary results, in particular those used in the proofs of Theorem~\ref{th:suff} and Lemma~\ref{lem:jump_conv}.

\subsection{Characteristic function of bounded distributions}

The following result characterizes distributions with bounded support in terms of their characteristic function. In the following, the characteristic function of a distribution $p$ is defined as  
\[
\varphi(t) = \mathbb{E}[e^{itX}], \quad X \sim p.
\]

A distribution \(p\), associated with a random variable \(X\), is said to have support bounded from below with lower bound \( a = \inf \operatorname{supp}(p) \) if, for any \( \varepsilon > 0 \),  
\[
p(\{X \leq a - \varepsilon\}) = 0 \quad \text{and} \quad p(\{X \leq a + \varepsilon\}) > 0.
\]

Similarly, \(p\) is said to have support bounded from above with upper bound \( b = \sup \operatorname{supp}(p) \) if, for any \( \varepsilon > 0 \),  
\[
p(\{X \leq b - \varepsilon\}) < 1 \quad \text{and} \quad p(\{X \leq b + \varepsilon\}) = 1.
\]

\begin{theorem}[Theorem 11.1.2 of \cite{lukacs1970characteristic}]
Let \( p \) be a distribution with characteristic function \( \varphi \). The distribution \( p \) has support bounded from below (respectively, from above) if and only $\varphi$ admits an analytic extension to the upper (respectively, lower) half-plane and satisfies the growth condition  
\[
\abs{\varphi(z)} \leq e^{C\abs{z}},
\]
for some constant \( C > 0 \) and for all $z\in\C$ such that \( \Im(z) > 0 \) (respectively, \( \Im(z) < 0 \)). The lower and upper bounds of \(p \) are given by:
\[
\inf\operatorname{supp}(p) = -\lim_{y\to\infty} \frac{\log \varphi(iy)}{y}, \quad
\sup\operatorname{supp}(p) = \lim_{y\to\infty} \frac{\log \varphi(-iy)}{y}.
\]
\end{theorem}

As in our work, it is often convenient to define the characteristic function over a complex domain. We therefore reformulate Theorem 11.1.22 of \cite{lukacs1970characteristic} as follows.

\begin{corollary}[Theorem 11.1.2 of \cite{lukacs1970characteristic}, reformulated]\label{cor:Lukacs}
Let \( p\) be a distribution function with characteristic function \( \varphi \) defined as  
\[
\varphi(u) = E[e^{uX}], \quad u \in i\mathbb{R}, \quad X \sim p.
\]
The distribution \( p \) has support bounded from below (respectively, from above) if and only if $\varphi$ admits an analytic extension to the left (respectively, right) half-plane and satisfies the growth condition  
\[
\abs{\varphi(w)} \leq e^{C\abs{w}},
\]
for some constant \( C > 0 \) and for all $z\in\C$ such that \( \Re(w) < 0 \) (respectively, \( \Re(w) > 0 \)). The lower and upper bounds of \( p \) are given by  
\[
\inf\operatorname{supp}(p) = -\lim_{y\to\infty} \frac{\log \varphi(-y)}{y}, \quad
\sup\operatorname{supp}(p) = \lim_{y\to\infty} \frac{\log \varphi(y)}{y}.
\]
\end{corollary}

\begin{proof}
We show only the case of a distribution bounded from below (the other case follows by the same argument). By applying the change of variable \( z \mapsto iz \), where \( w = iy \), the upper half-plane is rotated by \( \frac{\pi}{2} \), while preserving the modulus. This change does not affect the growth condition, and thus we obtain the same upper bound for \( \varphi(w) \) in the left half-plane.

The limit conditions for the lower and upper bounds of the distribution function \( F(x) \) follow naturally by applying the change of variable to \( z = iy \), with \( y > 0 \).
\end{proof}

\subsection{Support of continuos images of random variables}

\begin{lemma}\label{lem:supp_continuous_map}
Let $X$ be a random variable taking values in a topological space $S$ endowed 
with its Borel $\sigma$-algebra, and let 
$g:S \to T$ be a continuous mapping into another topological space $T$, also 
endowed with its Borel $\sigma$-algebra. Then,
\[
\operatorname{supp}(g(X)) \;=\; \operatorname{cl}\big(g(\operatorname{supp}(X))\big),
\]
where $\operatorname{cl}$ denotes the topological closure in $T$.
\end{lemma}

\begin{proof}
We prove the equality by showing the two inclusions.  For the first inclusion, let
\(P_X\) and \(P_Y\) denote the laws of \(X\) and \(Y\coloneqq g(X)\), respectively,
and take \(y\in\operatorname{supp}(g(X))\). By definition of the support, every
open neighbourhood \(V\) of \(y\) satisfies \(P_Y(V)>0\). Since \(P_Y\) is the pushforward of \(P_X\) under \(g\), we have
\[
P_X\big(g^{-1}(V)\big)=P_Y(V)>0.
\]
Moreover, because \(g\) is continuous, the preimage \(g^{-1}(V)\) of the open set
\(V\) is itself open in \(S\), and hence Borel measurable. Recall that \(\operatorname{supp}(X)\) is the complement of the largest open set
of \(P_X\)-measure zero, hence \(P_X(\operatorname{supp}(X))=1\). Therefore any
measurable set of positive \(P_X\)-measure must intersect \(\operatorname{supp}(X)\);
in particular \(g^{-1}(V)\cap\operatorname{supp}(X)\neq\varnothing\). Consequently
there exists \(x\in g^{-1}(V)\cap\operatorname{supp}(X)\), and hence
\(g(x)\in V\cap g(\operatorname{supp}(X))\). Since \(V\) is an arbitrary open
neighbourhood of \(y\), we conclude that \(y\in\operatorname{cl}\big(g(\operatorname{supp}(X))\big)\).
Thus
\[
\operatorname{supp}(g(X))\subseteq\operatorname{cl}\big(g(\operatorname{supp}(X))\big).
\]
To prove the reverse inclusion, let $y \in \operatorname{cl}\big(g(\operatorname{supp}(X))\big)$. 
By definition of closure, for every open neighborhood $U$ of $y$ in $T$, we have 
$U \cap g(\operatorname{supp}(X)) \neq \varnothing$. Hence, there exists 
$y' \in U \cap g(\operatorname{supp}(X))$, and therefore some $x' \in \operatorname{supp}(X)$ 
such that $y' = g(x') \in U$. Equivalently, $x' \in g^{-1}(U)$.  

Since $g$ is continuous, $g^{-1}(U)$ is open in $S$. Moreover, because $x' \in \operatorname{supp}(X)$, 
it follows that 
\[
P_X(g^{-1}(U)) > 0. 
\]
But $P_Y(U) = P_X(g^{-1}(U))$, and thus 
$P_Y(U) > 0$ for every open neighborhood $U$ of $y$. By the definition of support, this shows that 
$y \in \operatorname{supp}(g(X))$.  

Therefore,
\[
\operatorname{cl}\big(g(\operatorname{supp}(X))\big) \subseteq \operatorname{supp}(g(X)).
\]
Combining both inclusions yields the claimed equality.
\end{proof}

\end{document}